\documentclass{amsart}
\usepackage{mathrsfs}
\usepackage{cases}
\usepackage{amssymb,amscd,amsthm}
\usepackage{latexsym}
\usepackage{hyperref}
\usepackage{color}

\usepackage[noadjust]{cite}

\usepackage{mathtools}

\usepackage{subfigure}
\usepackage{graphicx}

\date{\today}

\newcommand{\Z}{{\mathbb Z}}
\newcommand{\R}{{\mathbb R}}
\newcommand{\C}{{\mathbb C}}
\newcommand{\D}{{\mathbb D}}
\newcommand{\T}{{\mathbb T}}

\newcommand{\N}{{\mathbb N}}

\newcommand{\set}[1]{\left\{#1\right\}}
\newcommand{\abs}[1]{\left| #1\right|}
\newcommand{\norm}[1]{\left\|#1\right\|}
\newcommand{\hdim}{\dim_{\mathrm{H}}}
\newcommand{\lhdim}{\dim_{\mathrm{H}}^\mathrm{loc}}
\newcommand{\bdim}{\dim_{\mathrm{B}}}
\newcommand{\lbdim}{\dim_{\mathrm{B}}^\mathrm{loc}}
\newcommand{\hdist}{\mathrm{dist}_\mathrm{H}}

\newtheorem{theorem}{Theorem} [section]
\newtheorem{remark}[theorem]{Remark}
\newtheorem{lemma}[theorem]{Lemma}
\newtheorem{prop}[theorem]{Proposition}
\newtheorem{coro}[theorem]{Corollary}
\newtheorem{definition}[theorem]{Definition}

\begin{document}

\title[OPUC with Fibonacci Verblunsky Coefficients]{Orthogonal Polynomials on the Unit Circle with Fibonacci Verblunsky Coefficients,\\ I.~The Essential Support of the Measure}

\author[D.\ Damanik]{David Damanik}

\address{Department of Mathematics, Rice University, Houston, TX~77005, USA}

\email {\href{mailto:damanik@rice.edu}{damanik@rice.edu}}

\urladdr {\href{http://www.ruf.rice.edu/~dtd3}{www.ruf.rice.edu/$\sim$dtd3}}

\author[P.\ Munger]{Paul Munger}

\address{Department of Mathematics, Rice University, Houston, TX~77005, USA}

\email{\href{mailto:pem1@rice.edu}{pem1@rice.edu}}

\urladdr{\href{http://pem1.web.rice.edu/}{http://pem1.web.rice.edu/}}

\author[W.\ N.\ Yessen]{William N.\ Yessen}

\address{Department of Mathematics, University of California, Irvine, CA~92697, USA}

\email{\href{mailto:wyessen@math.uci.edu}{wyessen@math.uci.edu}}

\urladdr{\href{http://sites.google.com/site/wyessen/}{http://sites.google.com/site/wyessen/}}

\thanks{D.\ D.\ was supported in part by NSF grants DMS--0800100 and DMS--1067988.}

\thanks{W. N. Y. was supported by NSF grant DMS-0901627, PI: A. Gorodetski}

\keywords{orthogonal polynomials, Fibonacci sequence, trace map}
\subjclass[2000]{Primary 42C05; Secondary 37D99}

\begin{abstract}
We study probability measures on the unit circle corresponding to orthogonal polynomials whose sequence of Verblunsky coefficients is invariant under the Fibonacci substitution. We focus in particular on the fractal properties of the essential support of these measures.
\end{abstract}

\maketitle

\section{Introduction}\label{sec:intro}

There is a well known one-to-one correspondence between probability measures on the unit circle and a class of five-diagonal matrices, the so-called CMV matrices. Let us recall this correspondence.

A CMV matrix is a semi-infinite matrix of the form
$$
\mathcal{C} = \begin{pmatrix}
{}& \bar\alpha_0 & \bar\alpha_1 \rho_0 & \rho_1
\rho_0
& 0 & 0 & \dots & {} \\
{}& \rho_0 & -\bar\alpha_1 \alpha_0 & -\rho_1
\alpha_0
& 0 & 0 & \dots & {} \\
{}& 0 & \bar\alpha_2 \rho_1 & -\bar\alpha_2 \alpha_1 &
\bar\alpha_3 \rho_2 & \rho_3 \rho_2 & \dots & {} \\
{}& 0 & \rho_2 \rho_1 & -\rho_2 \alpha_1 &
-\bar\alpha_3
\alpha_2 & -\rho_3 \alpha_2 & \dots & {} \\
{}& 0 & 0 & 0 & \bar\alpha_4 \rho_3 & -\bar\alpha_4
\alpha_3
& \dots & {} \\
{}& \dots & \dots & \dots & \dots & \dots & \dots & {}
\end{pmatrix}
$$
where $\alpha_n \in \D = \{ w \in \C : |w| < 1 \}$ and $\rho_n = (1-|\alpha_n|^2)^{1/2}$. $\mathcal{C}$ defines a unitary operator on $\ell^2(\Z_+)$.

CMV matrices $\mathcal{C}$ are in one-to-one correspondence to probability measures $\mu$ on the unit circle $\partial \D$ that are not supported by a finite set. To go from $\mathcal{C}$ to $\mu$, one invokes the spectral theorem. To go from $\mu$ to $\mathcal{C}$, one can proceed either via orthogonal polynomials or via Schur functions. In the approach via orthogonal polynomials, the $\alpha_n$'s arise as recursion coefficients for the polynomials.

Explicitly, consider the Hilbert space $L^2(\partial \D,d\mu)$ and apply the Gram-Schmidt orthonormalization procedure to the
sequence of monomials $1 , w , w^2 , w^3 , \ldots$. This yields a sequence $\varphi_0 , \varphi_1 , \varphi_2 , \varphi_3 , \ldots$ of normalized polynomials that are pairwise orthogonal in $L^2(\partial \D,d\mu)$. Corresponding to $\varphi_n$, consider the ``reflected polynomial'' $\varphi_n^*$, where the coefficients of $\varphi_n$ are conjugated and then written in reverse order. Then, we have
\begin{equation}\label{e.tmbasic}
\begin{pmatrix} \varphi_{n+1}(w) \\ \varphi_{n+1}^*(w) \end{pmatrix} = \rho_n^{-1} \left( \begin{array}{cc} w & - \bar{\alpha}_n \\ - \alpha_n w& 1 \end{array} \right) \begin{pmatrix} \varphi_{n}(w) \\ \varphi_{n}^*(w) \end{pmatrix}
\end{equation}
for suitably chosen $\alpha_n \in \D$ (and again with $\rho_n = (1-|\alpha_n|^2)^{1/2}$). With these $\alpha_n$'s, one may form the corresponding CMV matrix $\mathcal{C}$ as above and obtain a unitary matrix for which the spectral measure corresponding to the cyclic vector $\delta_0$ is indeed the measure $\mu$ we based the construction on.

Depending on whether one starts out with the coefficients or the measure in this one-to-one correspondence, one obtains direct and inverse spectral theory in this setting. In this paper we will start out with the coefficients and study the associated measure. The coefficients $\alpha_n$ will be chosen to be invariant under the Fibonacci substitution $S : a \mapsto ab$, $b \mapsto a$. The study of this particular case was begun by B.~Simon in Section~12.8 of \cite{S2}, where it was pointed out that it is a natural problem to pursue this analysis further. Given the recent advances on the analogous problem in the world of orthogonal polynomials on the real line \cite{C, DG, Y2}, it is now a good time to carry out this investigation.

Let us discuss the model in more detail. Given a sequence $\{ \alpha_n \}_{n \ge 0}$ of Verblunsky coefficients that take only two values $\alpha, \beta \in \D$, we view it as an element of $\{ \alpha , \beta \}^{\Z_+}$. The substitution $S$ may be extended by concatenation to $\{ \alpha , \beta \}^{\Z_+}$, that is, the one-sided infinite word $\omega_0 \omega_1 \omega_2 \ldots \in \{ \alpha , \beta \}^{\Z_+}$ is sent to the one-sided infinite word $S(\omega_0) S(\omega_1) S(\omega_2) \ldots \in \{ \alpha , \beta \}^{\Z_+}$. There is a unique fixed point, namely, $\omega_{\alpha,\beta} = \alpha \beta \alpha \alpha \beta \ldots \in \Omega$. It can be obtained by iterating $S$ on $\alpha$. Indeed, the sequence of finite words $\alpha = S^0(\alpha)$, $\alpha \beta = S^1(\alpha)$, $\alpha \beta \alpha = S^2(\alpha)$, $\alpha \beta \alpha \alpha \beta = S^3 (\alpha)$ clearly converges to an element of $\{ \alpha , \beta \}^{\Z_+}$ (in the sense that each of these finite words is a prefix of the limit word) and
any infinite word that is fixed under $S$ arises in this way. The associated CMV matrix will be denoted by $\mathcal{C}_{\omega_{\alpha,\beta}}$.

We will study the probability measure on the unit circle corresponding to the sequence of Verblunsky coefficients given by $\omega_{\alpha,\beta}$ for $\alpha , \beta \in \D$, that is, $\alpha_0 \alpha_1 \alpha_2 \alpha_3 \alpha_4 \ldots = \alpha \beta \alpha \alpha \beta \ldots$. Recall that the (topological) support of $\mu_{\omega_{\alpha,\beta}}$ is the smallest closed subset of $\partial \D$ so that its complement has zero measure with respect to $\mu_{\alpha,\beta}$. Removing isolated points from this set, we obtain the essential (topological) support of $\mu_{\omega_{\alpha,\beta}}$. We will denote this set by $\Sigma_{\alpha,\beta}$.\footnote{The set $\Sigma_{\alpha,\beta}$ is the spectrum of the natural two-sided extension $\mathcal{E}_{\alpha,\beta}$ of $\mathcal{C}_{\alpha,\beta}$, acting unitarily on $\ell^2(\Z)$.}

It is natural and often beneficial to embed these considerations in a subshift context. Given the fixed point $\mu_{\omega_{\alpha,\beta}}$ of $S$ defined above, let us denote by $\Omega_{\alpha,\beta} \subseteq \{ \alpha , \beta \}^{\Z_+}$ the set of all one-sided infinite words that agree locally with $\omega_{\alpha,\beta}$, that is, infinite words that have the same set of finite subwords as $\omega_{\alpha,\beta}$. The set $\Omega_{\alpha,\beta}$ is called the subshift generated by the substitution $S$ on the alphabet $\{ \alpha, \beta \}$. If we consider a sequence of Verblunsky coefficients following some $\omega \in \Omega_{\alpha,\beta}$, then the associated probability measure on the unit circle will be denoted by $\mu_\omega$ and the associated CMV matrix will be denoted by $\mathcal{C}_{\omega}$. The essential spectrum of $\mathcal{C}_{\omega}$, and hence the essential (topological) support of $\mu_{\omega}$, is independent of $\omega \in \Omega_{\alpha,\beta}$ and hence it is equal to $\Sigma_{\
alpha,\beta}$. That is, results for $\Sigma_{\alpha,\beta}$ will be of relevance for the OPUC problem associated with any $\omega \in \Omega_{\alpha,\beta}$. On the other hand, the measures $\mu_\omega$ themselves are not independent of $\omega \in \Omega_{\alpha,\beta}$ and hence a finer analysis of the properties of these measures will have to take this into account.

\section{Trace Map, Invariant, and the Curve of Initial Conditions}\label{sec:geometric-setup}

%%%%%%%%%%%
% Preliminary stuff
%%%%%%%%%%%

From the letters $\alpha$, $\beta$ of the alphabet define the constants $\rho = \sqrt{1- |\alpha|^2}$, $\sigma = \sqrt{1-|\beta|^2}$, and $K = 2(1-\mathrm{Re}(\overline{\alpha} \beta))$. The fundamental quantity is $x_n(w)$, half the trace of the $2\times2$ transfer matrix associated with the sequence $\omega_{\alpha,\beta}$ of Verblunsky coefficients across $f_n$ (the $n^{\mathrm{th}}$ Fibonacci number) sites:\footnote{We use the notation from \cite{S, S2} throughout. In particular, the transfer matrix $T_{k}(w;\mu)$ is given by a product of $k$ matrices of the form appearing in \eqref{e.tmbasic} and maps $(\varphi_{0}(w), \varphi_{0}^*(w) )^T$ to $(\varphi_{k}(w), \varphi_{k}^*(w) )^T$.}
$$x_n(w) = \frac{1}{2} w^{-f_n/2} \, \mathrm{Tr} \, T_{f_n}(w;\mu_\omega).$$
These traces obey the recursion
\begin{equation}\label{e.tracemaprecursion}
x_{n+1} = 2 x_n x_{n-1} - x_{n-2}
\end{equation}
if we define $x_0$ to be $(w^{1/2} + w^{-1/2})/(2\sigma)$ and $x_{-1}$ to be $K/(2\rho \sigma)$; compare \cite[Theorem~12.8.5]{S2}. The traces $x_n(w)$ are important because $w \in \Sigma_{\alpha,\beta}$ if and only if $x_n (w)$ is a bounded sequence (\cite[Theorem~12.8.3]{S2}).

The quantity
$$
I(w) = x^2_{n+1}(w) + x^2_n(w) + x^2_{n-1}(w) - 2 x_{n+1}(w) x_n(w) x_{n-1}(w) - 1
$$
depends on $w$ but not $n$; see \cite[Theorem~12.8.5]{S2}. By choosing $n=1$ we then get
\begin{align*}
I(w)  &= x^2 _1 (w) + x^2 _0 (w) + x^2 _{-1} (w) - 2 x_1 (w) x_0 (w) x_{-1} (w) - 1\\
& = K^2/4\rho^2 \sigma^2 + (w + 2 + w^{-1})/4\sigma^2 + (w + 2 + w^{-1})/4\rho^2 - \frac{K(w + 2 + w^{-1})}{4\rho^2 \sigma^2} - 1 \\
& = \mathrm{Re} \, w \left(\frac{1}{2\rho^2} + \frac{1}{2\sigma^2} - \frac{K}{2\rho^2 \sigma^2}\right) + \frac{K^2 - 2K}{4\rho^2 \sigma^2} + \frac{1}{2\sigma^2} + \frac{1}{2\rho^2} - 1.
\end{align*}

From this it is clear that $I(w)$ is a real number for all $w$ on the unit circle.  Considering the trace recursion as a map $T: \R^3 \ \longrightarrow \R^3$,\footnote{We denote a point in $\R^3$ by $(x,y,z)$ and for this reason use $w$ to denote the spectral parameter.}
$$
T(x,y,z) = ( 2 x y - z, x, y).
$$
The $n$-invariance of $I(w)$ comes from the fact that each of the sets
$$
S_V = \{(x, y, z)\ : \ x^2 + y^2 + z^2 - 2xyz - 1 = V\}
$$
is preserved by $T$. If $V > 0$, then $S_V$ is a smooth, connected, non-compact two-dimensional submanifold of $\R^3$ homeomorphic to the four-punctured sphere. When $V = 0$, $S_V$ develops four conic singularities, away from which it is smooth (see Figure \ref{fig:inv-surfaces} (a)). The surface $S_0$ is sometimes called \textit{the Cayley cubic}. When $-1 < V < 0$, $S_V$ contains five smooth connected components: four noncompact, homeomorphic to the two-disc, and one compact, homeomorphic to the two-sphere. When $V = -1$, $S_V$ consists of the four smooth noncompact discs and a point at the origin. When $V < -1$, $S_V$ consists only of the four noncompact two-discs. Compare with Figure \ref{fig:inv-surfaces}.

To investigate the dynamics of the trace map, we consider the curve of initial conditions
\begin{align*}
\gamma_{\alpha,\beta}(w) & = (x_1 (w), x_0 (w), x_{-1} (w)) = \left( \frac{w^{1/2} + w^{-1/2}}{2\rho}, \frac{w^{1/2} + w^{-1/2}}{2\sigma}, \frac{K}{2\rho \sigma} \right) \\
& =: \left( \frac{\eta}{2\rho}, \frac{\eta}{2\sigma}, \frac{K}{2\rho \sigma} \right)
\end{align*}
and ask how points on it behave under iteration of the map $T$.

Much of our analysis is based on dynamical properties of this so-called \textit{Fibonacci trace map} - an analytic map defined on $\R^3$. Its first appearance in the literature dates back to early 1980's, when a connection between smooth dynamical systems and spectral analysis of quasiperiodic Schr\"odinger operators was discovered and explored in the pioneering works of Kohmoto et al.\ \cite{kkt} and Ostlund et al.\ \cite{oprss}.

In this section we discuss some properties of the Fibonacci trace map, including some model-independent results. We then discuss the connection between the Fibonacci trace map and the essential support of the spectrum of CMV matrices with Fibonacci coefficients. This connection is exploited in Section \ref{sec:fractal-dims} to give a detailed description of the fractal nature of the essential support of the spectrum, as well as estimates on its fractal dimensions.

The existing literature on trace maps is extensive, including some rather comprehensive surveys. Since it is not our intention here to give a comprehensive overview of trace maps or of their connection with quasiperiodic Hamiltonians, known results are quoted only as necessary. We do, however, point the interested reader to \cite{C, BR, BR2, BR3, Rob} (and references therein) for a deeper look.

\subsection{Preliminary Results on Dynamics of the Fibonacci Trace Map}\label{subsec:dynamics-prelims}

Clearly $T$ is analytic, is defined on all of $\R^3$, and is invertible with the inverse
\begin{align}\label{eq:tmap-inverse}
T^{-1}(x,y,z) = (y,z, 2yz - x).
\end{align}
%

%
%
%
%%%%%% FIGURE: SURFACES %%%%%%%%%%%
%
%
\begin{figure}
\subfigure[$V = 0.0001$]{
\includegraphics[scale=.22]{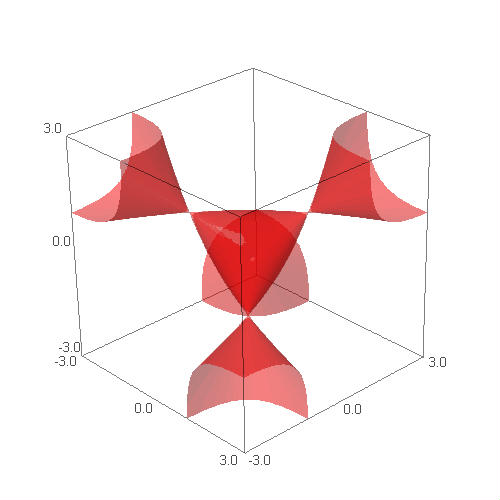}}
\subfigure[$V = 0.01$]{
\includegraphics[scale=.22]{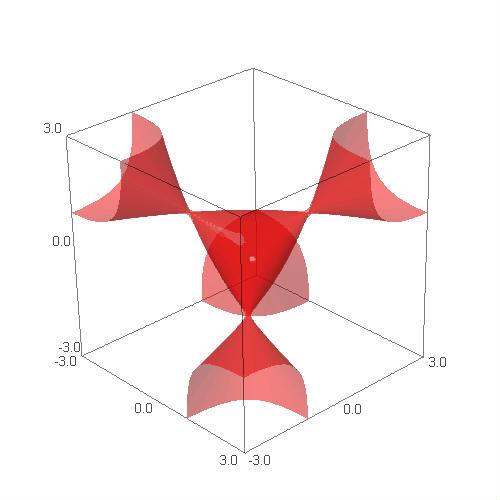}}
\subfigure[$V = 0.05$]{
\includegraphics[scale=.22]{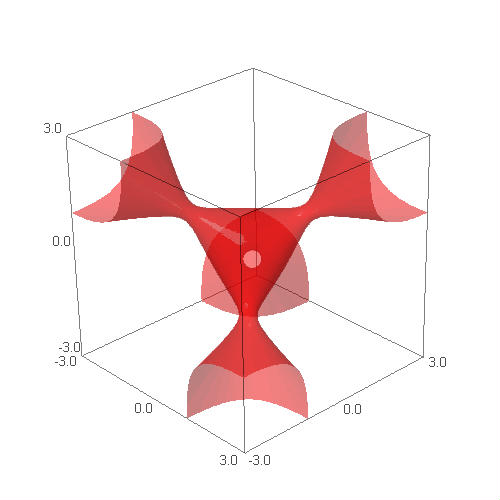}}
\\
\subfigure[$V = -0.95$]{
\includegraphics[scale=.22]{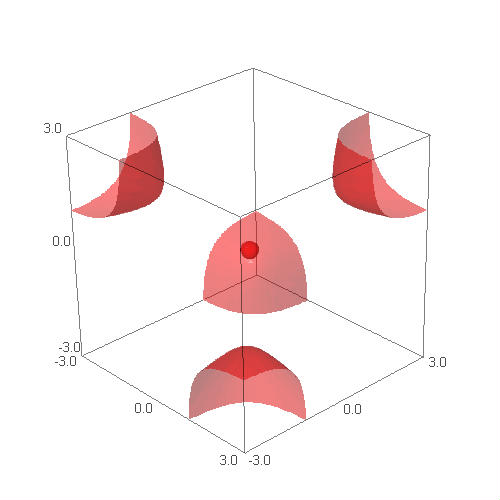}}
\subfigure[$V = -0.5$]{
\includegraphics[scale=.22]{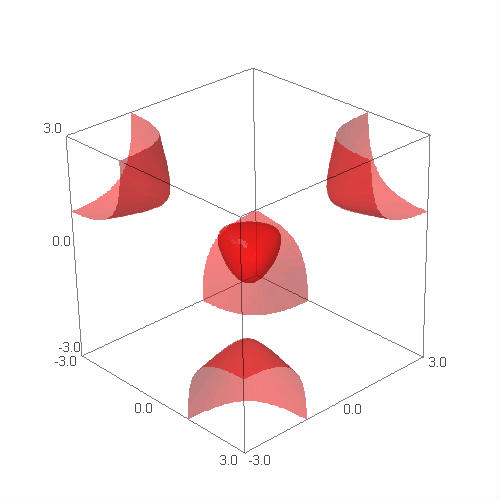}}
\subfigure[$V = -0.1$]{
\includegraphics[scale=.22]{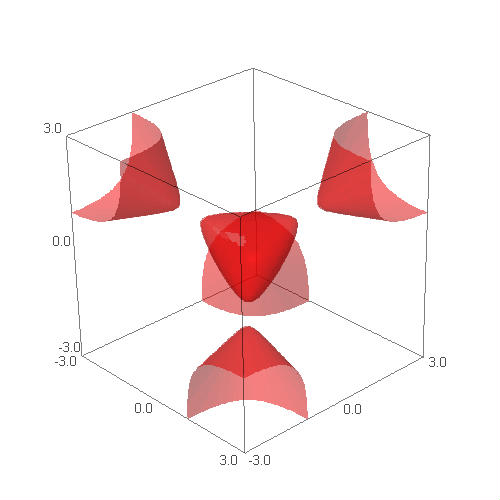}}
\caption{Invariant surfaces $S_V$ for four values of $V$ (figure taken from \cite{Y})}
\label{fig:inv-surfaces}
\end{figure}
%
%
%%%%%% END FIGURE %%%%%%%%%%%%%%%%%
%
%

Define a smooth three-dimensional submanifold $\mathcal{M}$ of $\R^3$ by
\begin{align}\label{eq:m}
\mathcal{M} := \bigcup_{V > 0}S_V.
\end{align}
We shall use this notation often.

In what follows, we shall be concerned with those points in $\mathcal{M}$, whose forward semi-orbit under $T$ is bounded; that is, $p\in\mathcal{M}$ that satisfy:
\begin{align*}
\mathcal{O}_T^+(p):= \set{T^n(p)}_{n\in\N}\text{ is bounded.}
\end{align*}
We similarly define the backward semi-orbit and the full orbit of a point $p$ by, respectively,
\begin{align*}
\mathcal{O}_T^-(p) = \set{T^{-n}(p)}_{n\in\N} \hspace{2mm}\text{ and }\hspace{2mm}\mathcal{O}_T(p) = \mathcal{O}_T^+(p)\bigcup\mathcal{O}_T^-(p).
\end{align*}

For convenience and brevity, we shall say that \textit{a point} $p$ \textit{satisfies property} $\mathbf{B}$ (or \textit{has property} $\mathbf{B}$, or \textit{is of type} $\mathbf{B}$, or \textit{is type}-$\mathbf{B}$) if $p$ has bounded forward semi-orbit.

We'll also need to identify points of type $\mathbf{B}$ on $S_0$. Since $T$ preserves $S_V$ for every $V$, and $\mathcal{M}$ is foliated by $\set{S_V}_{V > 0}$, our task is equivalent to identifying points of type $\mathbf{B}$ on $S_V$, for $V \geq 0$. We do this next.

We begin with a basic result that guarantees that if the forward semi-orbit of a point under $T$ is unbounded, then it does not contain any infinitely long bounded subsequences.

\begin{prop}[See \cite{Rob}]\label{prop:escape}
A point $p\in\R^3$ is either a type-$\mathbf{B}$ point for $T$, or its orbit under $T$ diverges to infinity superexponentially fast in every coordinate.
\end{prop}

In what follows, we use (standard) notation and terminology from the theory of hyperbolic dynamical systems. For a brief overview, see Appendix \ref{a1} below.

\subsubsection{Dynamics of $T$ on the Cayley Cubic $S_0$}\label{subsubsec:dynamics-cayley-cubic}

Let $\mathbb{T}^2$ denote the two-dimensional torus $\R^2/\Z^2$, and $\mathcal{A}:\mathbb{T}^2\rightarrow \mathbb{T}^2$ - an automorphism on $\mathbb{T}^2$ given by the matrix
\begin{align}\label{eq:torus-auto}
\mathcal{A} = \begin{pmatrix}1 & 1\\ 1 & 0\end{pmatrix}.
\end{align}
Observe that $\mathcal{A}$ induces an Anosov diffeomorphism on $\mathbb{T}^2$. Now, define $\mathcal{F}: \mathbb{T}^2\rightarrow\R^3$ by
\begin{align}\label{eq:semiconjugacy}
\mathcal{F}(\theta,\phi) = (\cos 2\pi(\theta + \phi), \cos 2\pi\theta, \cos 2\pi\phi).
\end{align}
Denote the part of the Cayley cubic that lies inside of the unit cube centered at the origin in $\R^3$ by $\mathbb{S}$. It turns out that $\mathbb{S}$ is invariant under $T$, and the map $\mathcal{F}$ defines a semiconjugacy between $(\mathbb{T}^2, \mathcal{A})$ and $(\mathbb{S}, T)$; that is, the following diagram commutes:
\begin{align}\label{eq:semiconjugacy-cd}
\begin{CD}
\mathbb{T}^2 @>\mathcal{A}>> \mathbb{T}^2\\
@V\mathcal{F}VV                       @VV\mathcal{F}V\\
\mathbb{S} @>T>>                      \mathbb{S}
\end{CD}
\end{align}
The map $\mathcal{F}$ is not, however, a conjugacy in the sense of \eqref{part2_eq3}, since $\mathcal{F}$ is not invertible. In fact, $(\mathbb{T}, \mathcal{F})$ is a double cover of $\mathbb{S}$.

By invariance of $\mathbb{S}$ under $T$ it follows that all points of $\mathbb{S}$ are of type $\mathbf{B}$. Let us now see whether there are any other type-$\mathbf{B}$ points on $S_0$.

As has been mentioned above, $S_0$ contains four conic singularities; explicitly, they are
\begin{align}\label{eq:singularities}
P_1 = (1,1,1),\hspace{2mm} P_2 = (-1,-1,1),\hspace{2mm} P_3 = (1,-1,-1),\hspace{2mm} P_4 = (-1,1,-1).
\end{align}
The point $P_1$ is fixed under $T$, while $P_2$, $P_3$ and $P_4$ form a three cycle:
\begin{align*}
P_1\overset{T}{\longmapsto}P_1;\hspace{4mm} P_2\overset{T}{\longmapsto}P_3\overset{T}{\longmapsto}P_4\overset{T}{\longmapsto}P_2
\end{align*}
(which can be verified via direct computation, or by the semiconjugacy \eqref{eq:semiconjugacy-cd}). As it will soon become apparent, it is convenient to work with $\widetilde{T}:= T^6$ (the six-fold iteration of $T$) instead of $T$. By Proposition \ref{prop:escape}, type-$\mathbf{B}$ points of $T$ are precisely type-$\mathbf{B}$ points of $\widetilde{T}$.

For each $i\in\set{1,\dots,4}$, there is a smooth curve $\rho_i$, containing no self-intersections, passing through the singularity $P_i$, such that $\rho_i\setminus{P_i}$ is a disjoint union of two smooth curves---call them $\rho_i^l$ and $\rho_i^r$---with the following properties:
\begin{itemize}

\item $\rho_i^{l,r}\subset \mathcal{M}$;

\item $T(\rho_1^l) = \rho_1^r$ and $T(\rho_1^r) = T(\rho_1^l)$. In particular, points of $\rho_1^{l,r}$ are periodic of period two, and $\rho_1$ is fixed under $T$, and hence also under $\widetilde{T}$;

\item The six curves $\rho_i^{l,r}$, $i = 2, 3, 4$, form a six cycle under $T$. In particular, points of $\rho_i^{l,r}$, $i = 2, 3, 4$, are periodic of period six, and hence for $i = 2, 3, 4$, $\rho_i$ is fixed under $\widetilde{T}$.

\end{itemize}
%
%
%
%%%%%% FIGURE: SURFACES %%%%%%%%%%%
%
%
\begin{figure}
\subfigure{
\includegraphics[scale=.28]{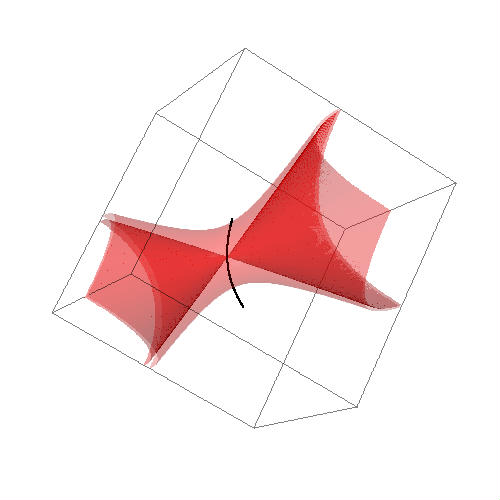}}
\subfigure{
\includegraphics[scale=.28]{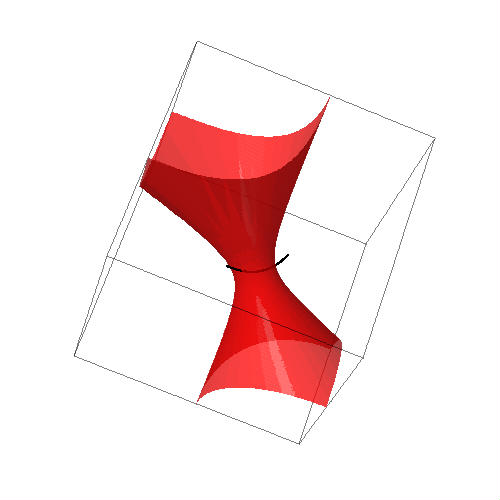}}
\caption{The curve $\rho_1$ in the vicinity of $P_1$ (figure taken from \cite{Y}).}
\label{fig:per}
\end{figure}
%
%
%%%%%% END FIGURE %%%%%%%%%%%%%%%%%
%
%

The curve $\rho_1$ (see Figure \ref{fig:per}) is given explicitly below, as we'll need this explicit expression later.
\begin{align}\label{eq:per-2-curve}
\rho_1 = \set{\left(x,\hspace{1mm}\frac{x}{2x-1},\hspace{1mm}x\right): x\in\left(-\infty,1/2\right)\cup\left(1/2,\infty\right)}.
\end{align}
Expressions for the other three curves can be obtained from \eqref{eq:per-2-curve} using symmetries of $T$ to be discussed below.

It follows via simple computation that for any $i = 1,\dots, 4$ and any point $p\in \rho_i^{l,r}$, the eigenvalue spectrum of $D\widetilde{T}_p$ is $\set{1, \lambda(p), 1/\lambda(p)}$ with $0 < \abs{\lambda(p)} < 1$, where $D\widetilde{T}_p$ denotes the differential of $\widetilde{T}$ at the point $p$. The eigenspace corresponding to the eigenvalue $1$ is tangent to $\rho_i$ at $p$. At $P_i$, the eigenvalue spectrum of $D\widetilde{T}_{P_i}$ is of the same form, and as above, the eigenspace corresponding to the unit eigenvalue is tangent to $\rho_i$ at $P_i$. It follows that the curves $\set{\rho_i}_{i=1,\dots,4}$ are normally hyperbolic one-dimensional submanifolds of $\R^3$, as defined in Section \ref{a_3}. This will be the main ingredient in the proof of

\begin{lemma}\label{lem:type-b-on-cones}
There exist type-$\mathbf{B}$ points in $S_0\setminus\mathbb{S}$; these points form a disjoint union of four smooth injectively immersed  connected one-dimensional submanifolds of $S_0\setminus \mathbb{S}$, $\mathcal{W}_1,\dots,\mathcal{W}_4$, such that for every $p\in \mathcal{W}_i$, we have
\begin{align}\label{eq:type-b-on-cones-eq1}
\lim_{n\rightarrow\infty}\widetilde{T}^n(p) = P_i,\hspace{2mm}\text{ and }\hspace{2mm}\widetilde{T}(\mathcal{W}_i) = \mathcal{W}_i.
\end{align}
Moreover, for $i\in\set{1,\dots,4}$, there exists an open neighborhood $\mathcal{U}$ of $P_i$, such that
\begin{align}\label{eq:type-b-on-cones-eq2}
\text{for every }
(x,y,z)\in\mathcal{W}_i\cap\mathcal{U}\hspace{2mm}\text{we have}\hspace{2mm}\abs{x},\abs{y},\abs{z} > 1.
\end{align}
In particular, if $p\in\mathcal{W}_i$, for any $i\in\set{1,\dots,4}$, then for all $n\in\N$ sufficiently large, all three coordinates of $\widetilde{T}^n$ (and hence of $T^n(p)$) are greater than one in absolute value.

Points of $S_0$ not belonging to $\mathbb{S}$ or $\bigcup_i \mathcal{W}_i$ are not of type $\mathbf{B}$.
\end{lemma}
(We shall need \eqref{eq:type-b-on-cones-eq2} later when we investigate fractal properties of the essential support of the spectra of CMV matrices.)

Before proving Lemma \ref{lem:type-b-on-cones}, we mention certain symmetries of the map $\widetilde{T}$ which can be employed to simplify technical details of some arguments. Indeed, it turns out that when proving geometric properties of $\widetilde{T}$, the singularities $\set{P_i}$ require special attention (see, for example, \cite{Damanik2009} and \cite{Y, Y2}). It turns out that it is enough to handle only $P_1$, due to certain symmetries of $\widetilde{T}$. For example, by applying these symmetries to $\rho_1$ in \ref{eq:per-2-curve}, one obtains existence and properties (discussed above) of the other three curves: $\rho_2$, $\rho_3$ and $\rho_4$. Let us discuss these symmetries now.

Let us denote the group of symmetries of $\widetilde{T}$ by $\mathcal{G}_\mathrm{sym}$, and the group of reversing symmetries of $\widetilde{T}$ by $\mathcal{G}_\mathrm{rev}$; that is,
\begin{align}\label{eq:sym-group}
\mathcal{G}_\mathrm{sym} = \set{s\in\mathrm{Diff}(\R^3): s\circ \widetilde{T}\circ s^{-1} = \widetilde{T}},
\end{align}
and
\begin{align}\label{eq:rev-group}
\mathcal{G}_\mathrm{rev} = \set{s\in\mathrm{Diff}(\R^3): s\circ \widetilde{T}\circ s^{-1} = \widetilde{T}^{-1}},
\end{align}
where $\mathrm{Diff}(\R^3)$ denotes the set of diffeomorphisms on $\R^3$.

Observe that $\mathcal{G}_\mathrm{rev}\neq\emptyset$. Indeed,
\begin{align}\label{eq:rev-sym}
s(x,y,z) = (z,y,x)
\end{align}
is a reversing symmetry of $T$, and hence also of $\widetilde{T}$. Hence $\widetilde{T}$ is smoothly conjugate to $\widetilde{T}^{-1}$. It follows (see Appendix \ref{a1}) that forward-time dynamical properties of $\widetilde{T}$, as well as the geometry of dynamical invariants (such as stable manifolds) are mapped smoothly and rigidly to those of $\widetilde{T}^{-1}$. That is, forward-time dynamics of $\widetilde{T}$ is essentially the same as its backward-time dynamics.

The group $\mathcal{G}_\mathrm{sym}$ is also nonempty, and more importantly, it contains the following diffeomorphisms:
\begin{align}\label{eq:symmetries}
s_2: (x,y,z)\mapsto (-x, -y, z),\notag\\
s_3: (x,y,z)\mapsto(x,-y,-z),\\
s_4: (x,y,z)\mapsto(-x,y,-z).\notag
\end{align}
Notice that the symmetries $\{s_i\}$ are rigid transformations. Also notice that
\begin{align}\label{eq:symmetry-P1}
s_i(P_1) = P_i.
\end{align}
For a more general and extensive discussion of symmetries and reversing symmetries of trace maps, see \cite{BR}.

We are now ready to prove Lemma \ref{lem:type-b-on-cones}.

\begin{proof}[Proof of Lemma \ref{lem:type-b-on-cones}]
Let us concentrate on the curve $\rho_1$ and, for statements near the singularities, on the singularity $P_1$. The general result will then follow by application of the symmetries $\set{s_i}$ from \eqref{eq:symmetries}.

As has already been discussed, $\rho_1$ is a fixed normally hyperbolic submanifold of $\R^3$ for the map $\widetilde{T}$. Moreover, $\rho_1\setminus\set{P_1} = \rho_i^l\cup\rho_i^l$ belongs to $\mathcal{M}$ ($\rho_1$ intersects $S_0$ only at the point $P_1$). Denote the stable set of $\rho_1$ by $W^s(\rho_1)$ (see Appendix \ref{a1} for definitions and notation), and let $W^{ss}(P_1)$ denote the one-dimensional stable manifold to $P_1$ in $W^s(\rho_1)$. By invariance of the surfaces $\set{S_V}$, it follows that $W^{ss}(P_1)\subset S_0$. Since $P_1$ is a conic singularity of $\mathbb{S}$ and $P_1\in W^{ss}(P_1)$, by smoothness of $W^{ss}(P_1)$, $W^{ss}(P_1)$ cannot be entirely contained in $\mathbb{S}$, and therefore has a part in the cone of $S_0$ that is attached to $P_1$. Denote this part by $\mathcal{W}_1$. Since points of $\mathcal{W}_1$ converge to $P_1$ in positive time, all points of $\mathcal{W}_1$ satisfy property $\mathbf{B}$. Moreover, $\mathcal{W}_1$ is a connected  smooth injectively-immersed one-
dimensional submanifold of $S_0\setminus\mathbb{S}$ (it is formed by the intersection of $W^s(\rho_1)$ with $S_0\setminus\mathbb{S}$, and this intersection is quadratic).

Next let us prove that all points of the cone of $S_0$ attached to $P_1$ other than those lying on $\mathcal{W}_1$ escape to infinity in forward time.

From the proofs of Propositions 5 and 6 in \cite{BR2}, it follows that a point $p$ in the given cone does not escape if and only if
\begin{align*}
\lim_{n\rightarrow\infty}\widetilde{T}^n(p) = P_1.
\end{align*}
Since $\rho_1$ contains $P_1$ and is normally hyperbolic, it follows that $p\in W^s(\rho_1)$, and hence $p\in \mathcal{W}_1$.

Conversely, if $p\in\mathcal{W}_1$, then since $p\in W^{ss}(P_1)$,
\begin{align*}
\lim_{n\rightarrow\infty}\widetilde{T}^n(p) = P_1.
\end{align*}
Invariance of $\mathcal{W}_1$ under $\widetilde{T}$ also follows immediately from normal hyperbolicity. This proves \eqref{eq:type-b-on-cones-eq1}.

To prove \eqref{eq:type-b-on-cones-eq2} in a neighborhood of $P_1$, observe that the tangent space to $\mathcal{W}_1$ at $P_1$ is spanned by the eigenvector of $DT_{P_1}$ corresponding to the eigenvalue which is smaller than one in absolute value. A simple computation shows that the $z$ component of this vector is positive. Combining this with the fact that the cone attached to the singularity $P_1$ does not intersect the unit cube centered at the origin at points other than $P_1$, we get that $\mathcal{W}_1$ in a neighborhood of $P_1$ must lie in the region where all three coordinates are greater than one.

Finally, all analogous claims for the other singularities follow by the symmetries $\set{s_i}_{i=2,3,4}$ defined in \eqref{eq:symmetries}, since these symmetries are rigid and $s_i$ maps the cone attached to $P_1$ to the one attached to $P_i$; moreover, all four cones are fixed under $\widetilde{T}$. This completes the proof.
\end{proof}

\subsubsection{Dynamics of $T$ on $S_{V < 0}$}\label{subsubsec:dynamics-inside}

As was mentioned earlier, the surface $S_V$, $V < 0$, consists of five connected components, one of which is bounded. It turns out that the bounded component is invariant under $T$, hence consists entirely of type-$\mathbf{B}$ points; on the other hand, every point on the non-bounded components escapes to infinity (see \cite{Rob}). This leads to the following simple result which will be used later in Section \ref{sec:fractal-dims}.

\begin{lemma}\label{lem:open-bounded}
For all $V_0 < 0$ and $p \in S_{V_0}$ of type $\mathbf{B}$, there exists an open neighborhood $B_p \subset \bigcup_{V < 0} S_V \subset{\mathbb{R}^3}$ of $p$, such that every point of $B_p$ is also a type-$\mathbf{B}$ point.
\end{lemma}

\begin{proof}
If $p \in S_{V_0}$, ${V_0 < 0}$, is of type $\mathbf{B}$, then $p$ belongs to the bounded component of $S_V$. Since the bounded components depend continuously (in fact, analytically, since $I$ is analytic) on $V$, it follows that any sufficiently small neighborhood of $p$ is contained entirely in the bounded component of $\bigcup_{V_1 < V < V_2}S_V$, with $V_i < 0$ and $V_0 \in [V_1, V_2]$. In fact, the bounded components of $S_V$, $V < 0$, form a smooth two-dimensional foliation of the bounded three dimensional manifold obtained by taking their union.
\end{proof}

\subsubsection{Dynamics of $T$ on $S_{V > 0}$}\label{subsubsec:dynamics-away}

In the rest of this section, when we write $S_V$, we implicitly assume that $V > 0$. We'll also write $T_V$ for $T|_{S_V}$, and similarly $\widetilde{T}_{V}$ for $\widetilde{T}|_{S_V}$.

The dynamics of the trace map $T_V$ (or, equivalently, of the map $\widetilde{T}_V$) is rather complex. This complexity arises from the fact that $T_V$ satisfies Smale's Axiom A - a hallmark of chaos (see \cite{Smale1967} for the origins of this terminology). Indeed, we have

\begin{theorem}[M. Casdagli; D. Damanik and A. Gorodetski; S. Cantat\footnote{The special case of $V\geq 16$ was done by M. Casdagli in \cite{Casdagli1986}. D. Damanik and A. Gorodetski extended the result to all $V > 0$ sufficiently small in \cite{Damanik2009}. Finally, S. Cantat proved the result for all $V > 0$ in \cite{C}. (D. Damanik and A. Gorodetski, and S. Cantat obtained their results independently, and used different techniques.)}]\label{thm:hyperbolicity}

The set of all bounded {\rm (}forward and backward{\rm )} orbits of $S_V$ under $T_V$ coincides with the nonwandering set $\Lambda_V$. The set $\Lambda_V$ is a compact locally maximal $T_V$-invariant hyperbolic subset of $S_V$. The periodic points of $T_V$ form a dense subset of $\Lambda_V$. Topologically, $\Lambda_V$ is a Cantor set. There exists a point in $\Lambda_V$ whose forward semi-orbit is dense in $\Lambda_V$ {\rm (}i.e., $T_V$ is transitive on $\Lambda_V${\rm )}.
\end{theorem}

A few remarks are in order here, before we continue. First, by \textit{nonwandering set} $\Lambda_V\subset S_V$ we mean the set of those points $p\in S_V$ that satisfy the following property. For every open neighborhood $U$ of $p$ and for any $N\in\N$, there exists $n > N$ such that $T_V^n(U)\cap U\neq\emptyset$. Second, Axiom A diffeomorphisms are those for which the nonwandering set is compact and hyperbolic, and periodic points form a dense subset of the nonwandering set. That the set of points with bounded orbits coincide with the nonwandering set is not part of Axiom A requirements (for instance, such a requirement would force every Axiom A diffeomorphism on a compact manifold to be Anosov; that is, hyperbolic on the entire manifold).

Now suppose that $p\in S_V$ is a type-$\mathbf{B}$ point. Then there exists a point $\widetilde{p}\in S_V$, such that $\lim_{i\rightarrow\infty}T^{n_i}(p) = \widetilde{p}$, with $n_i\in\N$, $n_i\uparrow\infty$. The point $\widetilde{p}$ is easily seen to be nonwandering. In fact, all limit points of $\mathcal{O}^+_T(p)$ are nonwandering. It follows that for any open neighborhood $U$ of $\Lambda_V$, for all sufficiently large $n\in\N$, $T^n(p)\in U$. Since $\Lambda_V$ is locally maximal, take $U$ as in \eqref{part2_eq1}. Now suppose that for infinitely many $k\in \N$, $T^k(p)\notin U$. Then there exists a limit point of $\mathcal{O}_T^+(p)$ outside of $U$, hence outside of $\Lambda_V$, which cannot be. Hence there exists $N_0\in\N$ such that for all $k\geq N_0$, $T^k(p)\in U$; in particular, $T^{N_0 + 1}(p)\in U$, and so $T^{N_0}(p)\in T^{-1}(U)$. By induction, for all $k\in\N$, $T^{N_0}(p)\in T^{-k}(U)$. On the other hand, $\bigcap_{k\in\N}T^{-k}(U)$ is precisely $W^s(\Lambda_V)\cap U$ (see \eqref{part2_eq5}
for the definition of $W^s(\Lambda_V)$)---this follows from the general theory (see, for example, the references given in the opening sentence of Section \ref{a_1}). It follows that $p$ must belong to the stable manifold of some $q\in\Lambda_V$. In the notation of Appendix \ref{a1}, we have $p\in W^s(q)$.

Conversely, if $p\in W^s(q)$ for some $q\in\Lambda_V$, then by definition of the stable manifold, we have $\lim_{n\rightarrow\infty}\norm{T^n(p) - T^n(q)} = 0$. Since $q\in\Lambda_V$, $\mathcal{O}_T^+(q)$ is bounded, hence $p$ is of type $\mathbf{B}$. In other words, we have proved
\begin{coro}\label{cor:stable-manifold}
A point $p$ is a type-$\mathbf{B}$ point in $S_V$ if and only if there exists $q\in\Lambda_V$, such that $p\in W^s(q)$.
\end{coro}
Observe that as a consequence the set of type-$\mathbf{B}$ points on $S_V$ carries the following geometry. It is a disjoint union of smooth one-dimensional injectively immersed connected submanifolds of $S_V$.

The result of Corollary \ref{cor:stable-manifold} and the corresponding geometry of type-$\mathbf{B}$ points have been applied in a number of papers investigating spectra of quasiperiodic Hamiltonians, starting with the pioneering work of M. Casdagli in \cite{Casdagli1986} and the following work of A. S\"ut\H{o} in \cite{S87} (see also \cite{DG} and references to earlier works therein). Let us remark that an explicit proof of Corollary \ref{cor:stable-manifold} isn't found in the aforementioned papers,  so, while it easily follows from general principles, we included one here for completeness.

While it was enough in those papers to consider the dynamics of $T_V$ for a fixed value of $V > 0$ and the corresponding type-$\mathbf{B}$ points, in our case, much like in \cite{Y2}, we have to consider type-$\mathbf{B}$ points in $\mathcal{M}$; that is, type-$\mathbf{B}$ points of $S_V$, for all $V > 0$ at once.

\subsubsection{Geometry of Type-$\mathbf{B}$ Points in $\mathcal{M}$}\label{subsubsec:type-b-geometry}

Let us now discuss the geometry of type-$\mathbf{B}$ points in $\mathcal{M}$. We have already seen above that on $S_V$, type-$\mathbf{B}$ points form a disjoint union of injectively immersed smooth one-dimensional connected submanifolds of $S_V$. Since $\mathcal{M}$ is smoothly foliated by the surfaces $\set{S_V}_{V > 0}$, it is natural to inquire whether the aforementioned one-dimensional manifolds form a meaningful geometric structure in $\mathcal{M}$, when viewed simultaneously for all $V>0$. To this end we have the following result, that originally appeared in \cite{Y}.

Before stating the theorem, we need to define
\begin{align}\label{eq:whole-omega}
\Lambda := \bigcup_{V > 0}\Lambda_V.
\end{align}
\begin{theorem}\label{thm:cs-manifolds}
There exists a family, denoted by $\mathcal{W}^s$, of smooth $2$-dimensional connected injectively immersed submanifolds of $\mathcal{M}$, whose members we denote by $W^{cs}$ and call \upshape center-stable manifolds, \itshape with the following properties.
\begin{enumerate}

\item The family $\mathcal{W}^s$ is $T$-invariant; that is, for any $W^{cs}\in\mathcal{W}^s$, $T(W^{cs})\in\mathcal{W}^s$;

\item For every $x\in\Lambda$, there exists a unique $W^{cs}(x)\in\mathcal{W}^s$ containing $x$;

\item Conversely, for every $W^{cs}\in\mathcal{W}^s$, there exists {\rm (}in fact many!{\rm )} $x\in\Lambda$ such that $x\in W^{cs}$;

\item For any $V > 0$ and any $W^{cs}\in\mathcal{W}^s$, $W^{cs}\cap S_V$ is precisely the stable manifold, $W^s(q)$, for some $q\in\Lambda_V$;

\item For every $W^{cs}\in\mathcal{W}^s$, $W^{cs}$ intersects $S_V$ for every $V > 0$, and this intersection is transversal {\rm (}though the angle of intersection may depend on $V$ and on the points along $S_V\cap W^{cs}${\rm )};

\item The type-$\mathbf{B}$ points of $\mathcal{M}$ are precisely $\bigcup_{W\in\mathcal{W}^s} W$.

\end{enumerate}
\end{theorem}

Notice in particular that statement (6) of Theorem \ref{thm:cs-manifolds} describes completely the geometry of type-$\mathbf{B}$ points of $\mathcal{M}$: it is a disjoint union of smooth two-dimensional injectively immersed connected submanifolds of $\mathcal{M}$.

A detailed proof of this theorem appears in \cite{Y}. For completeness we give here a rough sketch of main ideas.

\begin{proof}[Proof of Theorem \ref{thm:cs-manifolds} {\rm (}sketch{\rm )}]
By the fundamental results of hyperbolic dynamics that are discussed in Appendix \ref{a1}, for any $V_1 > V_0 > 0$, the dynamics of $T|_{\Lambda_{V_0}}$ is conjugate to that of $T|_{\Lambda_{V_1}}$. That is, there exists a homeomorphism $H_{V_0, V_1}$ such that the following diagram commutes:
\begin{align*}
\begin{CD}
\Lambda_{V_1} @>T>> \Lambda_{V_1}\\
@VH_{V_0,V_1}VV                       @VVH_{V_0,V_1}V\\
\Lambda_{V_0} @>T>>                      \Lambda_{V_0}
\end{CD}
\end{align*}
Moreover, $H_{V_0,V_1}$ is a unique homeomorphism with this property.

To prove existence and smoothness of the center-stable manifolds, by the results of \cite[Section 6]{Hirsch1977} it is enough to show that for any $V_1 > V_0 > 0$,
\begin{align*}
\bigcup_{V_0\leq V\leq V_1}\Lambda_V
\end{align*}
is partially hyperbolic (see Appendix \ref{a1} for definitions). But this follows since the dynamics of $T$ on
\begin{align*}
\bigcup_{V_0\leq V\leq V_1}S_V
\end{align*}
is smoothly conjugated to a skew product of a hyperbolic diffeomorphism on a surface with the identity map on an interval (for details, see \cite{Y2}).

To prove that the center-stable manifolds are transversal to the surfaces $\set{S_V}_{V > 0}$, notice that, for any fixed $x\in \Lambda_{V_0}$, and any $V\neq 0$, the curve
\begin{align*}
V\mapsto H_{V_0, V}(x)
\end{align*}
is smooth (indeed, it is the intersection of the center-stable manifold containing $x$, with the center-unstable manifold containing $x$, where the center-unstable manifolds are defined analogously to the center-stable ones for $T^{-1}$). Hence it is enough to show transversality of this curve with the invariant surfaces. This follows, since the homeomorphism $H_{V_0, V}$ depends Lipschitz-continuously on $V$ (see \cite{Y2} for details).
\end{proof}

\begin{figure}\label{Wcs}
\input{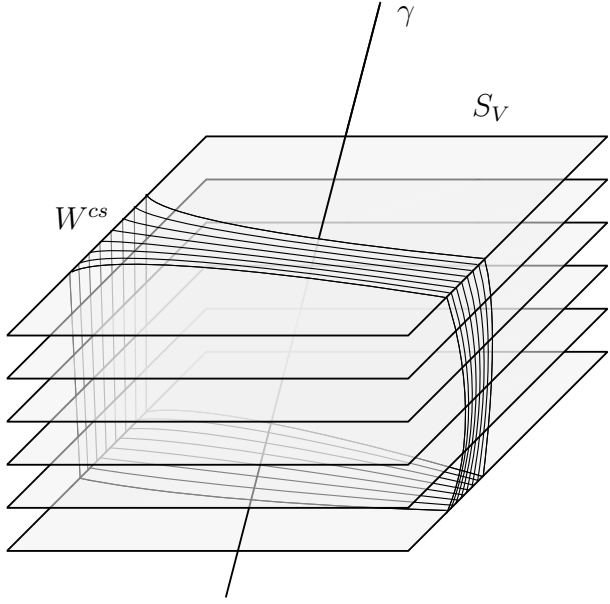}
\caption{Schematic illustration of the invariant surfaces, the curve of initial conditions, and the center-stable manifolds.}
\label{figure:wcs}
\end{figure}

The following (at first glance rather unmotivated) result will play a role in the spectral analysis of CMV matrices later.

\begin{prop}\label{prop:Barry-points}
There are no type-$\mathbf{B}$ points $(x,y,z)\in S_V$, such that for all $k$, all three coordinates of $T^k(x,y,z)$ remain greater than one in absolute value {\rm (}recall that we are assuming $V > 0${\rm )}.
\end{prop}

\begin{proof}
For convenience, let us call the region of interest $\mathcal{R}$:
\begin{align*}
\mathcal{R} = \set{(x,y,z)\in\R^3: \abs{x},\abs{y},\abs{z}> 1}.
\end{align*}
We wish to show that if $p$ is a type-$\mathbf{B}$ point in $S_V$, then for some $k\geq 1$, $T^k(p)\in \mathcal{R}^c$.

Observe that, according to \cite[Corollary 4.1]{Rob}, $\mathcal{R}$ does not contain any periodic points of $S_V$. Since periodic points are dense in $\Lambda_V$, we have $\Lambda_V\subset \mathcal{R}^c$. Let us now prove that for every $(x,y,z)\in\Lambda_V$, at least one of the three components is strictly smaller than one in absolute value. This amounts to checking a few cases.

First observe that the singularities $\set{P_i}_{i=1,\dots,4}$ do not qualify for consideration, since they lie on the Cayley cubic $S_0$. Before we continue checking the other cases, let us mention the following sufficient condition for escape, which we shall use here and later.

\begin{lemma}[{Sufficient Condition for Escape---see \cite[Proof of Proposition 5]{BR2}}]\label{lem:escape-cond}

If $\abs{x},\abs{y} > 1$ and $\abs{xy}\geq \abs{z}$, then the point $(x,y,z)$ is not of type $\mathbf{B}$. Hence, by the reversing symmetry of $T$ {\rm (}see \eqref{eq:rev-sym}{\rm )}, if $\abs{y},\abs{z}>1$ and $\abs{yz}\geq\abs{x}$, then $(x,y,z)$ escapes under the action of $T^{-1}$.
\end{lemma}

Let us now consider the points $(1,1,-1)$, $(1,-1,1)$ and $(-1,1,1)$. After iterating each of the three points forward or backward, we get:
\begin{align*}
T^2(1,1,-1) = (5,3,1),\hspace{2mm} T^2(1,-1,1) = (-7,-3,1),\hspace{2mm} T^{-2}(-1,1,1) = (1,3,5),
\end{align*}
all satisfying sufficient condition for escape either in forward or in backward time. Hence these three points do not belong to $\Lambda_V$. It remains to check for points of the form $(a, b, \pm 1)$, $(a, \pm 1, b)$ and $(\pm 1, a, b)$, where at least one of $\abs{a}, \abs{b}$ is larger than one, and neither is smaller than one.

If $\abs{a},\abs{b}>1$, then $(a,b,\pm 1)$ and $(\pm 1, a, b)$ cannot belong to $\Lambda_V$ by Lemma~\ref{lem:escape-cond}. Similarly, depending on whether $\abs{a}\geq\abs{b}$ or $\abs{b} \geq \abs{a}$, either $T(a, \pm 1, b)$ or $T^{-1}(a, \pm 1, b)$ escapes by Lemma~\ref{lem:escape-cond}, so $(a, \pm 1, b)\notin\Lambda_V$.

It remains to check for points of the form $(a, \pm 1, \pm 1)$, $(\pm 1, a, \pm 1)$ and $(\pm 1, \pm 1, a)$ with $\abs{a}>1$. We omit the necessary computations here, remarking that one follows exactly the same procedure as above.

Now, let $U$ be a neighborhood of $\Lambda_V$ such that every point of $U$ has at least one coordinate strictly smaller than one in absolute value. Say $p$ is type-$\mathbf{B}$ in $S_V$. Since $p$ lies on the stable manifold to some $q\in\Lambda_V$, for all sufficiently large $n\in\N$, $T^n(p)\in U$. This completes the proof.
\end{proof}

\subsection{Model-Independent Implications}\label{subsec:model-independent}

In this section we give a generalized, model-independent discussion of techniques that have been useful in spectral analysis of quasiperiodic Schr\"odinger and Jacobi operators (and now also CMV matrices). Indeed, this generalization is motivated by a rather persistent geometric scheme (see \cite{Casdagli1986}, \cite{DG} and references therein, \cite{Y2} and references therein). We then apply the results of this section to derive a topological, measure-theoretic and fractal-dimensional description of the essential support of the spectra of CMV matrices with Fibonacci Verblunsky coefficients.

\subsubsection{Dynamical Spectrum: Definitions and Basic Results}\label{subsubsec:ds-basic}

Given a subset $D$ of $\R^3$, we define the \textit{dynamical spectrum} of $D$, denoted by $B_\infty(D)$, as the set of those points of $D$ that have property $\mathbf{B}$:
\begin{align}\label{eq:ds}
B_\infty(D) := \set{p\in D: p\text{ is a type-}\mathbf{B}\text{ point}}.
\end{align}
In the rest of this section and in the next section, we shall be concerned with topological, measure-theoretic and (rather nontrivial) fractal-dimensional properties of dynamical spectra of compact analytic curves in $\mathcal{M}$.

For the rest of this section (and, in fact, for the rest of the paper) we assume that $\gamma$ is a compact analytic curve injectively immersed in $\R^3$.

We begin with the following fundamental result.

\begin{lemma}\label{lem:finite-tangencies}
Suppose $\gamma \subset \mathcal{M}$. If $\gamma$ does not lie entirely in a single center-stable manifold, then it has at most finitely many tangential intersections with the center-stable manifolds, while all other intersections are transversal.
\end{lemma}

\begin{proof}
This result follows by application of \cite[Lemma 6.4]{Bedford1993} to guarantee that tangential intersections are isolated and hence, by compactness of $\gamma$, there cannot be infinitely many of them. Indeed, to apply \cite[Lemma 6.4]{Bedford1993}, one only needs analyticity of $\gamma$ and of the center-stable manifolds; the former is one of our current hypotheses, while the latter will be proved in the forthcoming paper \cite{GorodetskiXXXX} (for a similar result in the context of Anosov diffeomorphisms, see \cite{delaLlave1986}).
\end{proof}

From Section \ref{subsec:dynamics-prelims}, we know that the dynamical spectrum of $\gamma$ is precisely the set of intersections of $\gamma$ with the center-stable manifolds:

\begin{align*}
B_\infty(\gamma) = \bigcup_{W^{cs}\in\mathcal{W}^s}\gamma\cap W^{cs}.
\end{align*}

If $\gamma$ lies entirely on a center-stable manifold, then $B_\infty(\gamma) = \gamma$. On the other hand, if $\gamma$ intersects center-stable manifolds only tangentially, then, by Lemma \ref{lem:finite-tangencies}, $B_\infty$ is a finite set. Also, $B_\infty$ is finite if $\gamma$ intersects the center-stable manifolds only at the endpoints of $\gamma$. In this case all is known about $B_\infty$. Our next result handles the other, far less trivial case:

\begin{theorem}\label{thm:cantor-geometric}
Suppose $\gamma$ contains a transversal intersection with a center-stable manifold away from the endpoints of $\gamma$. Then $B_\infty(\gamma)$ is a Cantor set, together with {\rm (}possibly{\rm )} finitely many isolated points. If $B_\infty(\gamma)$ contains these isolated points, then they necessarily arise as tangential intersections of $\gamma$ with the center-stable manifolds.
\end{theorem}

As an immediate corollary of Lemma \ref{lem:finite-tangencies} together with Theorem \ref{thm:cantor-geometric}, we obtain:

\begin{coro}\label{cor:cantor-geometric}
If $\gamma\subset\mathcal{M}$ does not lie entirely in a center-stable manifold, and if $B_\infty(\gamma)$ is infinite, then it is a Cantor set, together with {\rm (}possibly{\rm )} finitely many isolated points; these isolated points, if they exist, are necessarily points of tangency of $\gamma$ with the center-stable manifolds.
\end{coro}

\begin{remark}
Before we continue, let us remark that while isolated points necessarily arise as tangencies, a tangency does not necessarily produce an isolated point; we shall comment further on this in the proof of Theorem \ref{thm:cantor-geometric} below.
\end{remark}

\begin{proof}[Proof of Theorem \ref{thm:cantor-geometric}]

We first show that isolated points, if such exist, arise necessarily as tangential intersections. Indeed, suppose $p\in B_\infty(\gamma)$ is a point of transversal intersection of $\gamma$ with a center-stable manifold, and $p$ is not an endpoint of $\gamma$. Let us call the center stable manifold intersecting $\gamma$ at $p$, $\mathcal{W}(p)$. Assume also that $p\in S_{V_p}$, $V_p > 0$. Take a smooth curve $\tau\subset S_{V_p}$ passing through $p$ and transversal to $\mathcal{W}(p)\cap S_{V_p}$ (recall: $\mathcal{W}(p)\cap S_{V_p}$ is a stable manifold to some $q\in\Lambda_{V_p}$). Since the stable manifolds to points in $\Lambda_{V_p}$ form a continuous lamination (continuous in the sense of $C^1$-topology---see Section \ref{a_1_1}), $\tau$ intersects all stable manifolds transversally in a sufficiently small neighborhood $U_p$ of $p$, and hence is transversal to the center-stable manifolds in $U_p$.

With $\mathcal{W}^s$ denoting the family of center-stable manifolds, as in Theorem \ref{thm:cs-manifolds}, and $\mathcal{W}^s\cap U_p$ - the center-stable manifolds inside $U_p$, define the following map
\begin{align}\label{eq:points-to-cs}
\mathcal{G}: B_\infty(\tau)\cap U_p \longrightarrow\mathcal{W}^s\cap U_p
\end{align}
by
\begin{align*}
\mathcal{G}(x) = \mathcal{W}(x).
\end{align*}
Here $\mathcal{W}(x)$ is the part inside $U_p$ of the center-stable manifold that contains $x$.

The map $\mathcal{G}$ is continuous (see the discussion and references in Section \ref{a_1_1}) in the sense that there exists a family of smooth embeddings $\set{\mathcal{E}_{x}}_{x\in B_\infty(\tau)\cap U_p}$ of the unit disc $\mathbb{D}$ into $\R^3$, such that $\mathcal{E}_x(\mathbb{D}) = \mathcal{W}(x)$, and these embeddings depend continuously on $x$ in the $C^1$-topology (actually, in the $C^k$-topology for any $k\in\N$, but $k = 1$ is sufficient for our means). Since $\gamma$ intersects $\mathcal{W}(p)$ transversally, it follows that if $U_p$ is sufficiently small, for all $x\in B_\infty(\tau)\cap U_p$, $\gamma$ intersects $\mathcal{G}(x)$ (and hence \textit{all} of the center-stable manifolds inside $U_p$) transversally. It follows that the holonomy map
\begin{align}\label{eq:holonomy-along-cs}
\mathfrak{h}: B_\infty(\gamma)\cap U_p\longrightarrow B_\infty(\tau)\cap U_p
\end{align}
defined by projecting points along the center-stable manifolds is well-defined and is in fact a homeomorphism. On the other hand, since $\Lambda_{V_p}$ is a Cantor set (see Theorem \ref{thm:hyperbolicity}), it follows that $B_\infty(\tau)\cap U_p$ is also a Cantor set; hence
\begin{align}\label{eq:sentence}
B_\infty(\gamma)\cap U_p\hspace{2mm}\text{ contains neither isolated points nor connected components.}
\end{align}

Now let us show that away from isolated points, $B_\infty(\gamma)$ is a Cantor set.

Notice that the set of type-$\mathbf{B}$ points in $\mathcal{M}$ is a closed set. For details see, for example, \cite[Lemma 3.4]{Y}. Hence away from the finitely many isolated points, $B_\infty(\gamma)$ is compact. At points of transversal intersection, \eqref{eq:sentence} holds. Now, if $q$ is a point of tangential intersection (but not an isolated point in $B_\infty(\gamma)$), then on a sufficiently small neighborhood of $q$, $q$ is the only tangential intersection, by Lemma \ref{lem:finite-tangencies}. Hence for all points $p$ sufficiently close to $q$, \eqref{eq:sentence} holds, and we are done.
\end{proof}

As we mentioned earlier, the discussion above of dynamical spectra of analytic curves is inspired by a recurrent geometric scheme in spectral analysis of quasiperiodic (particularly, Fibonacci) Hamiltonians. It turns out that spectra of those Hamiltonians (or essential spectra of CMV matrices) correspond to dynamical spectra of analytic curves that satisfy the hypothesis of Theorem \ref{thm:cantor-geometric} (or, equivalently, of Corollary \ref{cor:cantor-geometric}). In this case, $B_\infty$ is a fractal, and more fine tuned analysis of $B_\infty$ is required. We do this next.

\subsubsection{Dynamical Spectrum: Fractal Dimensions}\label{subsubsec:ds-fd}

For the remainder of this section, we assume that $\gamma\subset\mathcal{M}$ satisfies the hypothesis of Theorem \ref{thm:cantor-geometric}.

The following three theorems give a further qualitative description of the fractal nature of $B_\infty(\gamma)$.

Before we continue, let us set up and fix for the remainder of this paper the following notation. The Hausdorff dimension of a set $A$ will be denoted by $\hdim(A)$. The local Hausdorff dimension of $A\subset\R$ at $a\in A$ is defined and denoted by
\begin{align*}
\lhdim(A, a) := \lim_{\epsilon\rightarrow 0}\hdim((a - \epsilon, a + \epsilon)\cap A).
\end{align*}
The box-counting dimension of $A$ (when it exists) is denoted by $\bdim(A)$, and the local box-counting dimension of $A\subset\R$ at $a\in A$ is defined analogously, and is denoted by $\lbdim(A, a)$. We denote the lower and upper box-counting dimensions of $A$ by, respectively, $\underline{\bdim}(A)$ and $\overline{\bdim}(A)$.

\begin{theorem}\label{thm:horseshoe-dim}
Take $a\in B_\infty(\gamma)$ and assume that $a\in S_{V_a}$. Suppose that $a$ is not an isolated point of $B_\infty(\gamma)$. Then
\begin{align}\label{eq:horseshoe-dim}
\lhdim(B_\infty(\gamma), a) = \frac{1}{2}\hdim(\Lambda_{V_a}),
\end{align}
where $\Lambda_{V_a}$ is the nonwandering set on $S_{V_a}$ from Theorem \ref{thm:hyperbolicity}.
\end{theorem}

As a consequence of the preceding theorem we obtain the following two results.

\begin{theorem}\label{thm:haus-dim}
The Hausdorff dimension of $B_\infty(\gamma)$ is strictly between zero and one. Consequently, the Lebesgue measure of $B_\infty(\gamma)$ is zero. If $\gamma$ lies entirely in some $S_V$, then for every $a\in B_\infty(\gamma)$, $\lhdim(B_\infty(\gamma), a) = \hdim(B_\infty(\gamma))$; otherwise:
\begin{enumerate}

\item $B_\infty(\gamma)\ni a\mapsto \lhdim(B_\infty(\gamma), a)$ is continuous;

\item For a non-isolated point $b\in B_\infty(\gamma)$ and any $\epsilon > 0$, the local Hausdorff dimension along $(b-\epsilon, b+\epsilon)\cap B_\infty(\gamma)$ is nonconstant.

\end{enumerate}
\end{theorem}

\begin{remark}
To be completely clear, let us remark that by $(b-\epsilon, b+\epsilon)$ we mean an $\epsilon$ neighborhood of $b$ along $\gamma$; equivalently, if $\gamma$ is parameterized on $[c, d]$ and $B\in (c, d)$ is such that $\gamma(B) = b$, then we can speak of $\gamma(B - \epsilon, B+\epsilon)$.
\end{remark}

Regarding the box-counting dimension of $B_\infty(\gamma)$, we have

\begin{theorem}\label{thm:box-dim}
Say $a\in B_\infty(\gamma)$ and $U_a$ a neighborhood {\rm (}in $\gamma${\rm )} of $a$, such that $\gamma\cap U_a$ intersects the center-stable manifolds transversally. Then the box-counting dimension of $U_a\cap B_\infty(\gamma)$ exists and $\bdim(U_a\cap B_\infty(\gamma)) = \hdim(U_a\cap B_\infty(\gamma))$.
\end{theorem}

Detailed proofs of these theorems appear in \cite{Y2} for the special case when $\gamma$ is a line. Those proofs carry over essentially verbatim to the presently considered general case. We outline the main ideas below.

\begin{proof}[Proof of Theorem \ref{thm:horseshoe-dim} {\rm (}outline{\rm )}]
Assume $p\in B_\infty(\gamma)$ is a point of transversal intersection. Let $S_{V_p}$, $\tau$ and $U_p$ be as in the proof of Theorem \ref{thm:cantor-geometric} above. It is proved in \cite{Y2} that the map $\mathfrak{h}$, as defined in \eqref{eq:holonomy-along-cs}, is H\"older continuous. Moreover, the H\"older exponent can be taken arbitrarily close to one by making $U_p$ sufficiently small. It follows that the local Hausdorff dimension at $p$ of $B_\infty(\gamma)$ coincides with that of $B_\infty(\tau)$. On the other hand,
\begin{align*}
\lhdim(B_\infty(\tau), p) = \frac{1}{2}\hdim(\Lambda_{V_p}).
\end{align*}
(The last equality follows from results in hyperbolic dynamics; for details see, for example, \cite{DG, Casdagli1986, C}).

Now, if $q\in B_\infty(\gamma)$ is a point of tangential intersection which is not an isolated point of $B_\infty(\gamma)$, let $U_\epsilon$ be an $\epsilon$-neighborhood of $q$ along $\gamma$, such that for every $p\in B_\infty(\gamma)\cap U_\epsilon\setminus\set{q}$, $p$ is a point of transversal intersection (which, again, is made possible by Lemma \ref{lem:finite-tangencies}). We have
\begin{align*}
\lhdim(B_\infty(\gamma),q) &= \lim_{\epsilon\rightarrow 0}\hdim(B_\infty(\gamma)\cap U_\epsilon\setminus\set{q}) \\
&= \lim_{\epsilon\rightarrow 0}\left(\sup_{p\in B_\infty(\gamma)\cap U_\epsilon\setminus\set{q}}\set{\lhdim(B_\infty(\gamma), p)}\right) \\
&=\lim_{\epsilon\rightarrow 0}\left(\sup_{p\in B_\infty(\gamma)\cap U_\epsilon\setminus\set{q}}\set{\frac{1}{2}\hdim(\Lambda_{V_p})}\right).
\end{align*}
On the other hand, the map
\begin{align}\label{eq:dim-horseshoes}
V\mapsto \hdim(\Lambda_V)\hspace{2mm}\text{ is continuous }
\end{align}
(in fact smooth -- see \cite{Mane1990}, and in our case even analytic -- see \cite{C}). Hence
\begin{align*}
\lim_{\epsilon\rightarrow 0}\left(\sup_{p\in B_\infty(\gamma)\cap U_\epsilon\setminus\set{q}}\set{\frac{1}{2}\hdim(\Lambda_{V_p})}\right) = \frac{1}{2}\hdim(\Lambda_{V_q}),
\end{align*}
where $V_q$ is such that $q\in S_{V_q}$.
\end{proof}

\begin{proof}[Proof of Theorem \ref{thm:haus-dim} {\rm (}outline{\rm )}]
The first statement of the theorem follows from the fact that $\hdim(\Lambda_V)$, $V > 0$, is strictly between zero and one (see \cite{C}). If $\gamma$ lies entirely in some $S_V$, $V > 0$, then away from (finitely many) tangential intersections, $B_\infty$ forms a so-called \textit{dynamically defined Cantor set} (see \cite{Casdagli1986, DG}), one of the properties of which is independence of local Hausdorff dimension on the point (see \cite[Chapter 4]{Palis1993} for definitions and results).

Finally, (1) follows from \eqref{eq:dim-horseshoes}, while (2) follows from analyticity of the map in \eqref{eq:dim-horseshoes}, together with the fact that
\begin{align*}
\lim_{V\rightarrow 0}\hdim(\Lambda_V) = 1,
\end{align*}
while for $V > 0$, $\hdim(\Lambda_V)\in(0,1)$ (notice that $(b-\epsilon, b+\epsilon)\cap B_\infty(\gamma)$ contains limit points).
\end{proof}

The proof of Theorem \ref{thm:box-dim} is rather technical (even in outline form). We invite the reader to see \cite[Proof of Theorem 2.5]{Y2} for details.

So far we have been concentrating on $\gamma\in\mathcal{M}$. Occasionally, however, one needs to consider $\gamma$ with a point on $S_0$ which is of type $\mathbf{B}$ (for example, see \cite[Theorem 2.3-ii]{Y2}). In this case, we have the following addition to our existing results.

We have already classified all type-$\mathbf{B}$ points in $S_0$ in Section \ref{subsubsec:dynamics-cayley-cubic}, hence if $\gamma$ lies entirely in $S_0$, then $B_\infty(\gamma)$ can be easily described by appealing to the results of the aforementioned section. Now let us handle the case where $\gamma\subset \mathcal{M}\cup S_0$, but does not lie entirely in $S_0$.

\begin{theorem}\label{thm:point-on-s0}
Suppose $\gamma$ is a compact analytic curve in $\mathcal{M}\cup S_0$ with no self-intersections. Assume also that $\gamma$ does not lie entirely in $S_0$. Assume that $\gamma$ satisfies the hypothesis of Theorem \ref{thm:cantor-geometric} {\rm (}or, equivalently, those of Corollary \ref{cor:cantor-geometric}{\rm )}. Then $B_\infty(\gamma)$ is of type described in Theorem \ref{thm:cantor-geometric}. Moreover, if $\gamma\cap S_0$ contains type-$\mathbf{B}$ points, and at least one of these points is not an isolated point of $B_\infty(\gamma)$, then at that point the local Hausdorff dimension of $B_\infty(\gamma)$ is equal to one; hence in this case the global Hausdorff dimension of $B_\infty(\gamma)$ is also equal to one.
\end{theorem}

\begin{proof}[Proof of Theorem \ref{thm:point-on-s0}]
Suppose that $p\in B_\infty$ lies in $S_0$ and is not an isolated point of $B_\infty$. Since by assumption $\gamma$ does not lie entirely in $S_0$, analyticity of $\gamma$ and of the Fricke-Vogt invariant $I$ implies that $\gamma$ must intersect $S_0$ in at most finitely many points. Hence there exists a neighborhood of $p$, say $U_p$, such that for all $q\in U_p$, with $q\neq p$, $q\notin S_0$. Since $p$ is not an isolated point, $U_p\cap B_\infty(\gamma)$ contains a sequence $p_n$ converging to $p$, and Theorem \ref{thm:horseshoe-dim} applies; that is:
\begin{align*}
\lhdim(B_\infty(\gamma),p_n) = \frac{1}{2}\hdim(\Lambda_{V_n}),
\end{align*}
where $V_n > 0$ is such that $p_n\in S_{V_n}$. Hence $V_n\longrightarrow 0$, and by the results of \cite{DG} we conclude that
\begin{align*}
\lim_{n\rightarrow\infty}\lhdim(B_\infty(\gamma), p_n) = 1.
\end{align*}
Hence $\hdim(B_\infty(\gamma)\cap U_p) = 1$, and therefore $\hdim(B_\infty(\gamma)) = \lhdim(B_\infty(\gamma),p) = 1$.
\end{proof}

Let us conclude this section with the following result, which describes the dependence of the Hausdorff dimension of $B_\infty(\gamma)$ on $\gamma$.

\begin{theorem}\label{thm:haus-cont}
Suppose $\gamma$ is such that $B_\infty(\gamma)$ is nonempty and does not contain any isolated points. Then $\hdim(B_\infty(\gamma))$ depends continuously on $\gamma$ in the $C^1$-topology.
\end{theorem}

The proof of this theorem follows immediately from the previous discussion, in particular \eqref{eq:horseshoe-dim}. Let us only remark that the result fails if $B_\infty(\gamma)$ contains isolated points. Indeed, say $\gamma$ is such that $B_\infty(\gamma)$ contains only one point, which is a point of quadratic intersection of $\gamma$ with a center-stable manifold. Then the Hausdorff dimension of $B_\infty(\gamma)$ is obviously zero. On the other hand, arbitrarily small perturbations $\widetilde{\gamma}$ of $\gamma$ produce a Cantor set for $B_\infty(\widetilde{\gamma})$ of strictly positive Hausdorff dimension, uniformly bounded away from zero.

\subsubsection{Band Spectrum Approximation of the Dynamical Spectrum}\label{subsubsec:band-spec-approx}

By analogy with spectral band structure of periodic approximations to the quasiperiodic Hamiltonians, we can construct approximations to $B_\infty(\gamma)$. One advantage of this general geometric construction, is that we can prove that these "band spectra" do in fact converge in Hausdorff metric to the actual dynamical spectrum $B_\infty(\gamma)$. As a result, we can prove that the spectra of periodic operators that converge strongly to the quasiperiodic one, converge in Hausdorff metric to the spectrum of the quasiperiodic operator.

Throughout this section, we assume that $\gamma\subset\mathcal{M}\cup S_0$ is compact, analytic and contains no self-intersections.

For $p\in\gamma$, let the components of $p$ be denoted by $p_x$, $p_y$, and $p_z$. That is, $p = (p_x,p_y,p_z)$. Since $\gamma$ is compact, there exists $C > 1$ such that $\max_{p\in\gamma}\abs{p_z} < C$.

We assume that $B_\infty(\gamma)$ does not contain any isolated points (and is nonempty, of course). Let us define the $n$th approximant of $B_\infty(\gamma)$, or the \textit{band spectrum on level $n$}, by
\begin{align*}
\sigma_n(\gamma) = \set{p\in\gamma: \abs{\pi\circ T^n(p)}\leq C},
\end{align*}
where $\pi$ denotes projection onto the third coordinate.

Before we state the result, let us quickly recall the definition of Hausdorff metric on $2^\R$. For any $A, B\subset\R$, define the Hausdorff metric $\hdist(A, B)$ by
%
%(package mathtools used here for adjustment of subscripts).
\begin{align*}
\hdist(A, B) = \max\set{\adjustlimits\sup_{a\in A}\inf_{b\in B}\set{|a -b|},\hspace{1mm}\adjustlimits\sup_{b\in B}\inf_{a\in A}\set{|a - b|}}.
\end{align*}
Unless there is danger of confusion, we shall drop $\gamma$ and write simply $\sigma_n$ and $B_\infty$ for $\sigma_n(\gamma)$ and $B_\infty(\gamma)$.

The following theorem describes how $B_\infty$ is approximated by $\sigma_n$.

\begin{theorem}\label{thm:band-approx}
We have
\begin{align}\label{eq:band-approx}
B_\infty = \bigcap_{n\geq 1}\sigma_n\cup\sigma_{n+1}.
\end{align}
Moreover, if $\gamma$ satisfies the hypothesis of Theorem \ref{thm:cantor-geometric} {\rm (}or, equivalently, Corollary \ref{cor:cantor-geometric}{\rm )}, and $B_\infty$ does not contain any isolated points, then
\begin{align}\label{eq:band-conv}
\sigma_n\longrightarrow B_\infty\hspace{2mm}\text{ with respect to the Hausdorff metric.}
\end{align}
\end{theorem}

A special, model-specific case of this result appears in \cite{Y}.

\begin{proof}[Proof of Theorem \ref{thm:band-approx}] We begin with a lemma that characterizes a type-\textbf{B} point in terms of relative magnitudes of its components:

\begin{lemma}\label{lem:bounded-orbit}
Given $p = (p_x, p_y, p_z)\in\R^3$, and $C \geq 1$ is such that $\abs{p_z}\leq C$, then $p$ is not of type-$\mathbf{B}$ if and only if there exists $N\in\N$ such that if $(p_x^N, p_y^N, p_z^N)$ denotes $T^N(p)$, then
\begin{align}\label{eq:escape-cond}
\abs{p_z^N}\leq C\hspace{2mm}\text{ and }\hspace{2mm}\abs{p_x^N}, \abs{p_y^N} > C.
\end{align}
\end{lemma}

For the proof of the preceding lemma, see \cite[Proposition 5.2]{Damanik2000}, replacing 1 therein with $C$ of Lemma \ref{lem:bounded-orbit}.

In what follows, for a point $p = (p_x, p_y, p_z)\in\R^3$, denote by $p^N = (p_x^N, p_y^N, p_z^N)$ the point $T^N(p)$. Observe that if $p\in\gamma$ and for some $N\geq 1$, $p^N\notin \sigma_N\cup\sigma_{N+1}$, then $p^N$ satisfies \eqref{eq:escape-cond} and so $p$ is not of type $\mathbf{B}$. Conversely, if $p\in\gamma$ is such that for all $n$, $p\in\sigma_n\cup\sigma_{n+1}$, then for all $n$, $\abs{p_\alpha^n} \leq C$, for some $\alpha\in\set{x,y,z}$. Application of Proposition \ref{prop:escape} then guarantees that $p$ is of type $\mathbf{B}$. In other words, we have proved \eqref{eq:band-approx}.

Let us now prove \eqref{eq:band-conv}. Let $\epsilon > 0$ chosen arbitrarily, and let $C_1,\dots, C_m$ be open sets of radius not larger than $\epsilon$ covering the compact set $B_\infty$. It is enough to show that there exists $N\in \N$ such that for all $n\geq N$, $\sigma_n\cap C_i^c = \emptyset$ and $\sigma_n\cap C_i\neq\emptyset$ for all $i\in\set{1,\dots,m}$.

Certainly since every $p\in B_\infty^c$ is not of type $\mathbf{B}$, for each such $p$, by Proposition~\ref{prop:escape}, there exists $N_p\in\N$ such that for all $n\geq N$, $\abs{p_z^{N_p}}>1$. By compactness of $\gamma$, there exists a common $N\in\N$, such that for all $p\in\left(\bigcup_{i=1}^m C_i\right)^c$, and $n\geq N$, $\abs{p_z^N} > 1$; that is, $p\notin \sigma_n$. Hence for all $n\geq N$, $\sigma_n\cap C_i^c = \emptyset$ for all $i\in\set{1,\dots,m}$.

Let us now return to the curve of period-two periodic points going through $P_1 = (1,1,1)$, which we denoted by $\rho_1$ (see Section \ref{subsubsec:dynamics-cayley-cubic}). Observe that $\rho_1\cap S_V$, for any $V> 0$, consists of two period-two periodic points $p_V^l$ and $p_V^r$ with $T(p_V^l) = p_V^r$ and $T(p_V^r) = p_V^l$. Consequently these points belong to $\Lambda_V$, and to each of these points there is attached a stable manifold, say $W^s(p_V^{l,r})$. Each of these two manifolds is dense in the lamination of stable manifolds on $S_V$ (see \cite{C}). Recall also that $P_1$ is a cutpoint of $\rho_1$, dividing it into two smooth curves $\rho_1^l$ and $\rho_1^r$, each contained entirely in $\mathcal{M}$, with ${p_V^l} = \rho_l^l\cap S_V$ and $p_V^r = \rho_1^r\cap S_V$. These two curves are normally hyperbolic, as was discussed in Section \ref{subsubsec:dynamics-cayley-cubic}, and consequently the stable manifold attached to each of these curves forms a dense sublamination of the lamination by center-
stable manifolds. Let us call these manifolds $\mathcal{W}^l$ and $\mathcal{W}^r$. It follows that for each $i\in\set{1,\dots,m}$, there exist points $p_i^l$ and $p_i^r$ in $C_i\cap B_\infty$, with $p_i^{l,r}\in \mathcal{W}^{l,r}$. Observe that the points on $\rho_1$ are of the form
\begin{align*}
\left(x, \frac{x}{2x - 1}, x\right)\hspace{2mm}\text{ with }\hspace{2mm}x\in\left(-\infty,\frac{1}{2}\right)\cup\left(\frac{1}{2},\infty\right).
\end{align*}
Consequently, if $p_V^l = (x, x/(2x - 1), x)$, then $p_V^r = (x/(2x - 1), x, x/(2x - 1))$. If $x = 1$, then we get $P_1$, which does not interest us, since $P_1\in S_0$. Otherwise, either $\abs{x} < 1$ or $\abs{x/(2x - 1)} < 1$; that is, either the absolute value of the $z$-component of $p_V^l$, or that of $p_V^r$, is strictly smaller one. It follows that for all sufficiently large $n$, either the absolute value of the $z$-component of $T^n(p_i^l)$ or that of $T^n(p_i^r)$ is strictly smaller than one; hence either $p_l^l$ or $p_i^r$ belongs to $\sigma_n$. Therefor, for all sufficiently large $n$, $\sigma_n\cap C_i\neq \emptyset$.
\end{proof}

So far we have carried out a qualitative (albeit rather detailed) analysis. Certainly, quantitative results are desired (such as estimates on fractal dimensions); however, even in model-specific cases such results are rather scarce and are notoriously difficult to obtain. This, among other things, will be the focus of our attention in Section \ref{sec:fractal-dims}. We shall comment further on previous model-specific quantitative results, as well as provide relevant references to previous works, in that section.

\section{The Fractal Dimension of the Essential Support}\label{sec:fractal-dims}

In this section we apply the results from Section~\ref{sec:geometric-setup} to the special case of interest in this paper. That is, we consider the set $\Sigma_{\alpha,\beta}$ associated with the Fibonacci substitution on two elements of the open unit disk in $\C$ and investigate its local and global fractal dimension by relating these quantities to the associated curve of initial conditions and the general results for the Fibonacci trace map presented in the previous section.

Let us fix $\alpha,\beta \in \D$, $\alpha \not= \beta$, and consider the associated one-sided infinite word $\omega_{\alpha,\beta}$ over the alphabet $\{ \alpha , \beta \}$ that is invariant under the substitution $S(\alpha) = \alpha \beta$ and $S(\beta) = \alpha$.

As explained in the beginning of Section~\ref{sec:geometric-setup}, the relevant curve of initial conditions for the OPUC problem generated by Verblunsky coefficients  $\omega_{\alpha,\beta}$ is given by
$$
\gamma_{\alpha,\beta}(w) := (x_1 (w), x_0 (w), x_{-1} (w)) = \left(\frac{w^{1/2} + w^{-1/2}}{2\rho}, \frac{w^{1/2} + w^{-1/2}}{2\sigma}, \frac{K}{2\rho \sigma}\right),
$$
where $\rho = \sqrt{1- |\alpha|^2}$, $\sigma = \sqrt{1-|\beta|^2}$, and $K = 2(1-\mathrm{Re}(\overline{\alpha} \beta))$. %For the sake of brevity and convenience, let us set $\eta := w^{1/2} + w^{-1/2}$. Then the expression for $\gamma_{\alpha,\beta}$ becomes
%$$\gamma_{\alpha,\beta}(w) = \left(\frac{\eta}{2\rho}, \frac{\eta}{2\sigma}, \frac{K}{2\rho \sigma}\right).$$
We are interested in those points on the curve of initial conditions $\gamma_{\alpha,\beta}$, which are of type $\mathbf{B}$. The corresponding $w$'s form the set $\Sigma_{\alpha,\beta}$, which is the essential spectrum of $\mathcal{C}_\omega$ and the essential (topological) support of $\mu_{\omega}$ for every $\omega \in \Omega_{\alpha,\beta}$:

\begin{theorem}\label{thm:spec-bdd}
The set $\Sigma_{\alpha,\beta}$ equals $\{ w \in \partial \D : \gamma_{\alpha,\beta}(w) \text{ is a type-$\mathbf{B}$ point} \}$ and is a Cantor set of zero Lebesgue measure.
\end{theorem}

This was first proved in \cite[12.8]{S2} and has been further elucidated in Section~\ref{sec:geometric-setup}. An important remark is that the set of points satisfying (2) in \cite[Proposition 12.8.6]{S2} is empty: this is Proposition~\ref{prop:Barry-points}.

Since the set $\Sigma_{\alpha,\beta}$ is known to be a Cantor set of zero Lebesgue measure, it is clearly of interest to study quantities such as fractal dimensions. Section~\ref{sec:geometric-setup} provides all the necessary general results, so that we may now exploit those, along with quantitative information obtained in the study of the Schr\"odinger case.

We first note the following:

\begin{theorem}\label{theorem:hdisbc}
There is a finite set $F_{\alpha,\beta} \subseteq \Sigma_{\alpha,\beta}$ such that for $w \in \Sigma_{\alpha,\beta} \setminus F_{\alpha,\beta}$, $\lbdim(\Sigma_{\alpha,\beta}; w)$ exists and is equal to $\lhdim(\Sigma_{\alpha,\beta}; w)$.
\end{theorem}

\begin{proof}
Since the curve of initial conditions is analytic and it is contained in no single invariant surface, it may have only isolated tangencies with invariant surfaces. Therefore Theorem~\ref{thm:box-dim} applies.
\end{proof}

Local dimensions such as $\lbdim(\Sigma_{\alpha,\beta}; w)$ and $\lhdim(\Sigma_{\alpha,\beta}; w)$ will depend on the value the invariant $I$ takes at the point $w \in \Sigma_{\alpha,\beta}$. Recall from Section~\ref{sec:geometric-setup} that
$$
I(w) = \mathrm{Re} \, w \left(\frac{1}{2\rho^2} + \frac{1}{2\sigma^2} - \frac{K}{2\rho^2 \sigma^2}\right) + \frac{K^2 - 2K}{4\rho^2 \sigma^2} + \frac{1}{2\sigma^2} + \frac{1}{2\rho^2} - 1.
$$
Since the spectral parameter $w$ belongs to the unit circle, $I$ as an affine function of $\mathrm{Re} \, w$ takes its maximum/minimum at $\pm 1$. In particular, we have
$$
I(1) = \frac{1}{\rho^2} + \frac{1}{\sigma^2} + \frac{K^2 - 4K}{4\rho^2 \sigma^2} - 1,
$$
so that a short calculation shows
$$
I(1) < 0 \Leftrightarrow | \mathrm{Re}(\overline{\alpha} \beta) | < |\alpha \beta|.
$$
As a consequence, we see that $I$ may take negative values for $\alpha , \beta \in \D$ and $w \in \partial \D$ suitably chosen. This observation is of interest since in all previous studies of models derived from a sequence invariant under the Fibonacci substitution, negative values of the invariant never occurred. This potentially complicates the situation in the OPUC setting.

On the essential spectrum, however, the invariant will always be non-negative:

\begin{prop}\label{prop:neg-inv}
We have $I(w) \ge 0$ for every $w \in \Sigma_{\alpha,\beta}$.
\end{prop}

\begin{proof}
Assume there exists $w \in \Sigma_{\alpha,\beta}$ such that $I(w) < 0$. Since $w \in \Sigma_{\alpha,\beta}$, $\gamma_{\alpha,\beta} (w)$ is a type-$\mathbf{B}$ point due to Theorem~\ref{thm:spec-bdd}. By the assumption $I(w) < 0$ and continuity of $I$, the same will be true in a sufficiently small neighborhood (in $\partial \D$) of $w$. Thus, nearby points $w'$ have negative invariant and give rise to type-$\mathbf{B}$ points $\gamma_{\alpha,\beta}(w')$ as well (this follows by Lemma \ref{lem:open-bounded}). This shows that an open neighborhood of $w$ belongs to $\Sigma_{\alpha,\beta}$, contradicting the fact that $\Sigma_{\alpha,\beta}$ is a Cantor set.
\end{proof}

This shows that, while $I(w)$ may in principle take negative values, for a spectral analysis of the OPUC problem, it is sufficient to study the trace map dynamics on the invariant surfaces corresponding to non-negative values of the invariant. Hence the potential complication alluded to above is actually quite tame and does not cause any real challenges.

\medskip

Let us now establish estimates of the local Hausdorff and box-counting dimensions under the premise that $I(w)$ is large enough. Since we will deal in particular with points of the spectrum where $I$ is greater than $16$ or $4$, we give a simple condition for the existence of such points:

\begin{prop}\label{prop:largeI}
Suppose that $\alpha = 0$ and let $M > 0$ be given. Then if $\rho ^2 = \frac{1}{M+1}$, $M$ is the maximum value attained by $I$ on $\T$.
\end{prop}

\begin{proof}
Since $\alpha = 0$, $\sigma = 1$. Thus
$$
I(w) = \mathrm{Re} \, w \left( \frac{1}{2} - \frac{1}{2\rho ^2} \right) - \frac{1}{2} + \frac{1}{2\rho ^2}.
$$
This function attains its maximum at $\mathrm{Re}\ w = -1$, so that
$$
M = -\left( \frac{1}{2} - \frac{1}{2\rho ^2} \right) - \frac{1}{2} + \frac{1}{2 \rho ^2} = \frac{1}{\rho ^2} - 1,
$$
as required.
\end{proof}

From this proof, it is also evident that (when $\alpha = 0$) $I(1) = 0$, but that by taking $|\beta|$ close enough to $1$, the set of points on $\partial \D$ where $I$ is greater than any fixed $M$ can have Lebesgue measure arbitrarily close to $2 \pi$.

\begin{theorem}\label{theorem:HDloc-bound}
With the finite set $F_{\alpha,\beta} \subseteq \Sigma_{\alpha,\beta}$ from Theorem~\ref{theorem:hdisbc}, suppose that $w \in \Sigma_{\alpha,\beta} \setminus F_{\alpha,\beta}$. Denote $S_u (w) = 4\sqrt{I(w)} + 22$ and $S_l (w) = \frac{1}{2} \left( 2\sqrt{I(w)} - 4 + \sqrt{(2\sqrt{I(w)}-4)^2 -12} \right)$. Then,
$$
I(w) > 4 \quad \Rightarrow \quad  \hdim^{\mathrm{loc}}(\Sigma_{\alpha,\beta};w) = \lbdim(\Sigma_{\alpha,\beta}; w) \ge \frac{\log(1+\sqrt{2})}{\log (S_u (w))}
$$
and
$$
I(w) \ge 16 \quad \Rightarrow \quad \hdim^{\mathrm{loc}}(\Sigma_{\alpha,\beta};w) = \lbdim(\Sigma_{\alpha,\beta}; w) \le \frac{\log(1+\sqrt{2})}{\log S_l (w)}.
$$
\end{theorem}

\begin{proof}
Let us consider an arbitrary $w \in \Sigma_{\alpha,\beta} \setminus F_{\alpha,\beta}$. Set $V = I(w)$, so that $\gamma_{\alpha,\beta} (w) \in S_V$. By Theorem~\ref{thm:horseshoe-dim},
$$
\hdim^{\mathrm{loc}}(\Sigma_{\alpha,\beta};w) = \lbdim(\Sigma_{\alpha,\beta}; w) = \frac12 \hdim (\Lambda_V).
$$
The quantity $\frac12 \hdim (\Lambda_V)$ was estimated in \cite{DEGT} from above and below, namely by implementing Theorem~\ref{thm:horseshoe-dim} for the particular case of the curve $C$ of initial conditions in the Schr\"odinger case (with the coupling constant $\lambda$ chosen so that $V = \frac{\lambda^2}{4}$) and then to use periodic approximation to obtain the desired dimension estimates from scaling properties of the spectra of the periodic approximants. The upper and lower bounds obtained in \cite{DEGT} yield the estimates claimed in the present theorem via the connection just described; see Figure~\ref{f.paul} for an illustration. The assumption $\lambda > 4$ (resp., $\lambda \ge 8$) necessary for the lower (resp., upper) bound from \cite{DEGT} to hold translates via $I(w) = \frac{\lambda^2}{4}$ to the assumption $I(w) > 4$ (resp., $I(w) \ge 16$) in the present context.
\end{proof}

\begin{figure}\label{f.paul}
\input{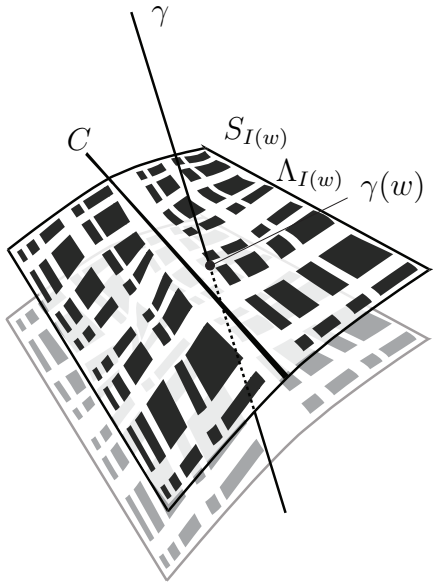}
\caption{An illustration of the idea behind the proof of Theorem \ref{theorem:HDloc-bound}.}
\label{figure:hdlocbd}
\end{figure}

\begin{remark}
Generically, the curve $\gamma$ of OPUC initial conditions and the curve $C$ of Schr\"odinger initial conditions never intersect. In fact, they only do so when $\alpha = \beta$ and the Schr\"odinger coupling constant $\lambda$ is $0$. In this case $\gamma$ is contained in $C$.
\end{remark}

\begin{proof}
The curve $C$ of Schr\"odinger initial conditions is given by $C(E) = ((E-\lambda)/2,E/2,1)$. Therefore $K = 2 \rho \sigma$ is a necessary condition for $\gamma(w)$ to lie on the curve $C$ for some $w$. Writing $\alpha = re^{i\theta}$ and $\beta = te^{i\psi}$, the equation $K = 2 \rho \sigma$ is the same as
$$
1-rt\cos(\psi-\theta)=\sqrt{(1-r^2)(1-t^2)}.
$$
Choosing arbitrary $\theta$, $r$, and $\psi$ we consider the equation as a quadratic polynomial in $t$:
$$
t^2(r^2\cos^2(\psi-\theta)+1-r^2) -t(2r\cos(\psi-\theta)) + r^2 =0.
$$
The equation has a real solution in $t$ when the discriminant is nonnegative:
$$
4(\cos^2(\psi-\theta)-1) - 4r^2(\cos^2(\psi-\theta)-1) \ge 0.
$$
This only happens when $\psi = \theta$. Verifying that this implies $r = t$ is straightforward, and it now follows that $\lambda = 0$ if $\gamma$ and $C$ intersect.
\end{proof}

As a consequence of Theorem~\ref{theorem:HDloc-bound}, we see that the local dimension of $\Sigma_{\alpha,\beta}$ near one of the points $w$ in question is asymptotic to $\frac{\log(1+\sqrt{2})}{\log \sqrt{I(w)}}$ in the large invariant regime. As far as the asymptotic behavior for small values of the invariant is concerned, we have the following result.

\begin{theorem}\label{theorem:HDloc-bound2}
There exists $I_0 > 0$ such that the following holds. With the finite set $F_{\alpha,\beta} \subseteq \Sigma_{\alpha,\beta}$ from Theorem~\ref{theorem:hdisbc}, suppose that $w \in \Sigma_{\alpha,\beta} \setminus F_{\alpha,\beta}$. Then,
$$
0 < I(w) < I_0 \quad \Rightarrow \quad  1 - \hdim^{\mathrm{loc}}(\Sigma_{\alpha,\beta};w) = 1 - \lbdim(\Sigma_{\alpha,\beta}; w) \simeq \sqrt{I(w)},
$$
where $a \simeq b$ means that $C^{-1} a \le b \le C a$ for some universal positive constant $C$.
\end{theorem}

\begin{proof}
This follows from \cite{DG} in the exact same way that Theorem~\ref{theorem:HDloc-bound} was derived from \cite{DEGT}.
\end{proof}

\section{Results for Individual Elements of the Subshift}\label{sec:elements}

Up to this point, the distinction between $\omega_{\alpha,\beta}$ and an arbitrary element $\omega$ of $\Omega_{\alpha,\beta}$ was not necessary as the essential spectrum of the associated CMV matrix is the same throughout the entire family. In this section we consider quantities that may indeed depend on the choice of $\omega \in \Omega_{\alpha,\beta}$. Recall that once such an $\omega$ is chosen, it provides a sequence of Verblunsky coefficients, and hence determines a probability measure on the unit circle and a CMV matrix acting on $\ell^2(\Z_+)$. On the other hand, in analogy to the known results in the Schr\"odinger context, one expects to be able to prove uniform results, that is, results that hold uniformly for all $\omega \in \Omega_{\alpha,\beta}$. This is a consequence of the rigid subword structure of the elements of the subshift, and more specifically, the fact that most subswords of Fibonacci length are cyclic permutations of the canonical word of this length from which the
corresponding trace in the trace map approach is derived.

\subsection{Absence of Point Masses Inside the Essential Support}

In this subsection we show that the restrictions of the measures $\mu_\omega$ to their essential support are purely continuous, that is, they have no point masses. The proof proceeds by showing that for any $w \in \sigma_\mathrm{ess}(\mu_{\omega})$ (which is equal to $\Sigma_{\alpha,\beta}$), the orthonormal polynomials $\varphi_n(w)$ associated with $\mu_{\omega}$ are not $\ell^2$ at the $w$ in question, which in turn is a consequence of the so-called Gordon Lemma and boundedness of transfer matrix traces. In fact, we prove this result for all the Aleksandrov measures derived from $\mu_\omega$, that is, uniformly in the boundary condition of the half-line problem.

Let us recall the basic connection between point masses and square-summability; compare \cite[Theorem~2.7.3]{S}.

\begin{lemma}\label{l.pmchar}
Let $\mu$ be a non-trivial probability measure on $\partial \D$ and let $w \in \partial \D$. Then $\mu(\{w\}) \not= 0$ if and only if $\sum_{n \ge 0} |\varphi_n(w;\mu)|^2 < \infty$.
\end{lemma}

Recall from \eqref{e.tmbasic} that
\begin{equation}\label{e.tmnstep}
\begin{pmatrix} \varphi_{n}(w;\mu) \\ \varphi_{n}^*(w;\mu) \end{pmatrix} = T_n(w;\mu) \begin{pmatrix} \varphi_{0}(w;\mu) \\ \varphi_{0}^*(w;\mu) \end{pmatrix} = T_n(w;\mu) \begin{pmatrix} 1 \\ 1 \end{pmatrix},
\end{equation}
where
\begin{equation}\label{e.tmproduct}
T_n(w;\mu) = A(\alpha_{n-1}(\mu),w) \cdots A(\alpha_0(\mu),w)
\end{equation}
and
\begin{equation}\label{e.tm1step}
A(\alpha,w) = (1-|\alpha|^2)^{-1/2} \left( \begin{array}{cc} w & - \bar{\alpha} \\ - \alpha w & 1 \end{array} \right).
\end{equation}

Recall also that the family of Aleksandrov measures $\{ \mu_\lambda : \lambda \in \partial \D \}$ associated with a non-trivial probability measure $\mu$ on $\partial \D$ is defined by
$$
\alpha_n(\mu_\lambda) = \lambda \alpha_n(\mu).
$$
Since we will be interested in the family of Aleksandrov measures associated with a measure generated by Fibonacci Verblunsky coefficients, we will need a version of \eqref{e.tmnstep}--\eqref{e.tm1step} for $\mu_\lambda$ instead of just $\mu$. It follows from \cite[Proposition~3.2.1]{S} that
\begin{equation}\label{e.tmnsteplambda}
\begin{pmatrix} \varphi_{n}(w;\mu_\lambda) \\ \varphi_{n}^*(w;\mu_\lambda) \end{pmatrix} = T_n(w;\mu) \begin{pmatrix} 1 \\ \bar \lambda \end{pmatrix}.
\end{equation}

We can now prove the main result of this subsection:

\begin{theorem}\label{t.esssuppcont}
For every $\omega \in \Omega_{\alpha,\beta}$ and $\lambda \in \partial \D$, we have $\mu_{\omega,\lambda}(\{w\}) = 0$ for every $w \in \sigma_\mathrm{ess}(\mu_{\omega,\lambda})$ {\rm (}$= \Sigma_{\alpha,\beta}${\rm )}.
\end{theorem}

\begin{proof}
Fix $\omega \in \Omega_{\alpha,\beta}$, $\lambda \in \partial \D$, and $w \in \sigma_\mathrm{ess}(\mu_{\omega,\lambda})$. If we can show that
\begin{equation}\label{e.vpnlnil2}
\sum_{n=0}^\infty |\varphi_{n}(w;\mu_{\omega,\lambda})|^2 = \infty,
\end{equation}
the result then follows from Lemma~\ref{l.pmchar}. Since $\varphi_{n}^*(w;\mu_{\omega,\lambda}) = w^n \overline{\varphi_{n}(w;\mu_{\omega,\lambda})}$, we have $|\varphi_{n}^*(w;\mu_{\omega,\lambda})| = |\varphi_{n}(w;\mu_{\omega,\lambda})|$ for every $n \ge 0$. In other words, \eqref{e.vpnlnil2} is equivalent to
\begin{equation}\label{e.vpnlnil2mod}
\sum_{n=0}^\infty |\varphi_{n}(w;\mu_{\omega,\lambda})|^2 + |\varphi_{n}^*(w;\mu_{\omega,\lambda})|^2 = \infty.
\end{equation}

Since $\sigma_\mathrm{ess}(\mu_{\omega,\lambda})$ does not depend on $\lambda$ (i.e., $\sigma_\mathrm{ess}(\mu_{\omega,\lambda}) = \sigma_\mathrm{ess}(\mu_{\omega,1}) = \Sigma_{\alpha,\beta}$) by \cite[Theorem~3.2.16]{S}, it follows from what we already know that for the special element $\omega_{\alpha,\beta}$ of $\Omega_{\alpha,\beta}$ that is invariant under the Fibonacci substitution, $|\mathrm{Tr} \; T_{F_k}(w;\mu_{\omega_{\alpha,\beta},1})|$ remains bounded as $k \to \infty$, say by $C$. It was shown in \cite{DL} that for the given $\omega$, there are infinitely many $k$ such that $\omega_{n+F_k} = \omega_n$ for $1 \le n \le F_k$ and $\omega_1 \ldots \omega_{F_k}$ is a subword of $\omega^*_1 \ldots \omega^*_{2F_k}$, that is, it is a cyclic permutation of $\omega^*_1 \ldots \omega^*_{F_k}$. By cyclic invariance of the trace, it follows that $|\mathrm{Tr} \; T_{F_k}(w;\mu_{\omega,1})| \le C$ for these infinitely many values of $k$.

Now, for such $k$'s, \eqref{e.tmnsteplambda} implies that
$$%\begin{equation}\label{e.gordon1}
\begin{pmatrix} \varphi_{F_k}(w;\mu_{\omega,\lambda}) \\ \varphi_{F_k}^*(w;\mu_{\omega,\lambda}) \end{pmatrix} = T_{F_k}(w;\mu_{\omega,1}) \begin{pmatrix} 1 \\ \bar \lambda \end{pmatrix}
$$%\end{equation}
and
$$%\begin{equation}\label{e.gordon2}
\begin{pmatrix} \varphi_{2F_k}(w;\mu_{\omega,\lambda}) \\ \varphi_{2F_k}^*(w;\mu_{\omega,\lambda}) \end{pmatrix} = T_{2F_k}(w;\mu_{\omega,1}) \begin{pmatrix} 1 \\ \bar \lambda \end{pmatrix} = T_{F_k}(w;\mu_{\omega,1})^2 \begin{pmatrix} 1 \\ \bar \lambda \end{pmatrix},
$$%\end{equation}
where we also used \eqref{e.tmproduct}. By the Cayley-Hamilton Theorem, this gives
$$
\max \left\{ \left\| \begin{pmatrix} \varphi_{F_k}(w;\mu_{\omega,\lambda}) \\ \varphi_{F_k}^*(w;\mu_{\omega,\lambda}) \end{pmatrix} \right\| , \left\| \begin{pmatrix} \varphi_{2F_k}(w;\mu_{\omega,\lambda}) \\ \varphi_{2F_k}^*(w;\mu_{\omega,\lambda}) \end{pmatrix} \right\| \right\} \ge \frac{1}{2C \sqrt{2}}.
$$
Since we have this estimate for infinitely many $k$, \eqref{e.vpnlnil2mod} follows.
\end{proof}

Clearly, the support of $\mu_{\omega,\lambda}$ is the disjoint union of $\sigma_\mathrm{ess}(\mu_{\omega,\lambda})$ and the isolated points outside this set that have non-zero weight with respect to $\mu_{\omega,\lambda}$. These isolated points are clearly point masses for the measure, while Theorem~\ref{t.esssuppcont} shows that $\mu_{\omega,\lambda}$ is purely singular continuous on the zero-measure set $\sigma_\mathrm{ess}(\mu_{\omega,\lambda})$ (assuming the non-trivial case where $\alpha,\beta$ are not equal). It was shown by Simon in \cite{S2} that for every $\omega \in \Omega_{\alpha,\beta}$ and Lebesgue almost every $\lambda \in \partial \D$, the measure $\mu_{\omega,\lambda}$ is pure point. That is, combining the two results, we see that it can happen that one of these two components has zero weight, and in fact it often does so.

Another consequence of Theorem~\ref{t.esssuppcont} is that the natural two-sided extension of any non-trivial Fibonacci CMV matrix has purely singular continuous spectrum. This follows since the two-sided extension has purely essential spectrum, which is equal to the essential spectrum of the one-sided matrix.

\subsection{Growth Estimates for the Orthogonal Polynomials}

Equation~\eqref{e.tmnsteplambda} shows that growth (and decay) properties of the normalized orthogonal polynomials are intimately connected to growth properties of the transfer matrices. For $n$ a Fibonacci number and $w \in \Sigma_{\alpha,\beta}$, the traces of these matrices are bounded in $n$ by an $I(w)$-dependent bound. The recursion~\eqref{e.tracemaprecursion} can then be used to derive norm estimates from these trace estimates, as observed by Iochum-Testard \cite{IT} and Damanik-Lenz \cite{DL3}. In this subsection, we follow this approach and derive the resulting estimates.

First we need some extra notation. Define $M(n,m,\omega,w)$ to be the transfer matrix at $w$ across $\omega \in \Omega_{a,b}$ restricted to the subword ranging from place $m$ to place $n$, multiplied by $w^{-|m-n|/2}$, so that $w^{-n/2} T_n(w; \mu_\omega) = M(n,0,\omega,w) $. Further, define $M_n(w) = M(f_n,0,x,w)$, where $x$ is the standard Fibonacci word upon which $\Omega_{a,b}$ is based. Finally let $r(w)$ be the largest root of the polynomial $x^3 - (2+2\sqrt{I(w)})x -1$. It is given by
$$r(w) = \frac{(27 + \sqrt{729-864(1+\sqrt{I(w)})^3})^{1/3}}{3 \cdot 2^{1/3}} + \frac{2 \cdot 2^{1/3} (1+\sqrt{I(w)})} {(27 + \sqrt{729-864(1+\sqrt{I(w)})^3})^{1/3}}.$$

\begin{theorem}\label{growth}
If $w \in \Sigma_{\alpha, \beta}$ and
$$
\gamma > \frac{\log((5+4\sqrt{I(w)})^{1/2} (3+2\sqrt{I(w)})r(w))}{\log\frac{\sqrt{5} + 1}{2}},
$$
then for a suitable $C$, we have that
$$
\sup_{\omega \in \Omega_{\alpha,\beta}} \norm{T_n (w;\mu_\omega)} \le Cn^{\gamma}
$$
for every $n \in \Z_+$; and therefore $|\varphi_n(w;\mu_\omega)| \le Cn^{\gamma}$.
\end{theorem}

As pointed out above, we follow the approach developed in the Schr\"odinger case by Iochum-Testard \cite{IT} and Damanik-Lenz \cite{DL3} and adapt it to the OPUC setting considered in this paper. The main idea is to use the trace estimates for canonical building blocks along with a partition of general finite subwords appearing in subshift elements in terms of the canonical building blocks in order to derive norm estimates for the general finite subwords.

The canonical building blocks are the prefixes of the special element $\omega_{\alpha,\beta}$ of the subshift $\Omega_{\alpha,\beta}$ that have length given by a Fibonacci number. The transfer matrices over these building blocks have traces bounded in absolute value by a constant that depends on the value of the invariant at the spectral parameter $w$. This dependence of the bound on $w$ motivates a quantitative study based on the value of $I(w)$. In this context we would like to mention also the paper \cite{DG} which worked out a quantitative version of \cite{IT} and \cite{DL3} in the Schr\"odinger case.

\begin{lemma}
\begin{enumerate}

\item If $w \in \Sigma_{\alpha, \beta}$, $|x_n(w)| \le 1 + \sqrt{I(w)}$ for all $n$.

\item For some positive $C(w)$ and for all $n \in \N$, $w \in \Sigma_{\alpha, \beta}$,
$$\norm{M_n(w)} \le C(w) r(w)^n .$$
Moreover, $\max_{w \in \Sigma_{\alpha, \beta}} C(w) =: C$ exists, and likewise for $r(w)$, so that
$$ \norm{M_n(w)} \le C r^n .$$

\end{enumerate}
\end{lemma}

\begin{proof}
This proof mimics the proof of \cite{DG}[Lemma 5.3].

We will prove part (b). Recall that $M_n (w)$ is the matrix $w^{-f_n /2} T_{f_n} (w)$, so that $x_n(w) = \frac{1}{2} \mathrm{tr} M_n (w)$. The Cayley-Hamilton theorem gives $M_n(w) ^2 +2x_n (w) M_n (w) +I = 0 $, or $M_n = 2x_n I - M_n ^{-1} $. Using the recursion for $M_n$,
$$M_n = M_{n-2} M_{n-1} = M_{n-2} (2x_{n-1} I - M_{n-1} ^{-1}) = 2x_{n-1} M_{n-2} - M_{n-3} ^{-1}. $$
Since $\norm{M_n} = \norm{M_n ^{-1}}$, the above yields
$$\norm{M_n (w)} = (2+2\sqrt{I(w)}) \norm{M_{n-2}(w)} + \norm{M_{n-3}(w)}, $$
for all $w \in \Sigma$.

Compare this estimate to the recursion
$$m_n = (2+2\sqrt{I})m_{n-2} + m_{n-3} $$
with initial conditions $m_1 = \norm{M_1 (w)}$, $m_2 = \norm{M_2 (w)}$, $m_3 = \norm{M_3 (w)}$, so that $\norm{M_n} \le m_n$. Any solution to the recursion is of the form
$$m_n = c_1(w) r_1 (w) ^n + c_2(w) r_2 (w) ^n + c_3(w) r_3 (w) ^n, $$
where the $r_j$ are the roots of the characteristic polynomial $x^3 - (2+2\sqrt{I})x -1$ to the recursion.

One can compute
$$c_1 = -\frac{-\norm{M_2} + \norm{M_1}r_2 + \norm{M_1}r_3 - \norm{M_0}r_2 r_3}{r_1(r_2 - r_1)(r_3-r_1)},$$
and similar expressions for $c_2$ and $c_3$. Taking $r_1 =r$ to be the largest root, the function
$$C(w) = c_1(w) + \frac{c_2 (w) r_2(w)}{r_1(w)} + \frac{c_3(w) r_3(w)}{r_1(w)}$$
satisfies
$$ \norm{M_n(w)} \le C(w) r(w)^n ,$$
proving the first part of the lemma. One can check that none of the roots $r_j (w)$ is ever zero, so that $C$ is a continuous function of $w$. The second part of the lemma now follows from compactness of $\Sigma_{\alpha, \beta}$.
\end{proof}

\begin{lemma}\label{lemma:add}
For all $n \ge 1$, $k \ge 0$,
$$M_n M_{n+k} = P_k ^{(1)} M_{n+k} + P_k ^{(2)} M_{n+k-1} +P_k ^{(3)} M_{n+k-2} + P_k ^{(4)} I,$$
where $P_k ^{(j)}$ are scalars depending on $n$ and $w$. Moreover, for every $n \ge 1$ and $w \in \Sigma$,
$$\sum^4_{j=1} |P_k ^{(j)} (w) | \le (5 + 4\sqrt{I(w)})(3+2\sqrt{I(w)})^{\lfloor k/2 \rfloor}.$$
\end{lemma}

The proof of the OPRL version of this lemma carries over from \cite{DG} with no change.

\begin{lemma}\label{mbound}
If $I(w) > 0$ and $\gamma$ obeys the hypothesis of Theorem~\ref{growth}, there is a constant $K$ such that
$$\norm{T_n(w)} \le K n^{\gamma}$$
for every $n\ge1$.
\end{lemma}

The proof of this lemma also carries over from \cite{DG}. One first writes $n$ as a sum of Fibonacci numbers, in the so-called Zeckendorf representation, to obtain $\|T_n\| = \|M_{n_0} M_{n_1} \dots M_{n_K}\|$. The construction is such that $F_{n_K} \le n < F_{n_K + 1}$ and $F_{2(K-1)} \le n$. Inducting on $K$, we can apply the estimate from Lemma~\ref{lemma:add} to get $\norm{M_{n_0} M_{n_1} \dots M_{n_{K}}} \le C(5+4\sqrt{I})^{K} \sqrt{3+2\sqrt{I}}^{2(K-1) - n_0} (r \sqrt{3+2\sqrt{I}})^{n_{K}}$. It follows that
\begin{align*}
&\limsup_{n\to\infty} \frac{\log\norm{T_n}}{\log n} \\
& \le \limsup_{n\to\infty} \frac{\log\left(  C(5+4\sqrt{I})^{K} \sqrt{3+2\sqrt{I}}^{2(K-1) - n_0} (r \sqrt{3 + 2\sqrt{I}})^{n_{K}}\right)}{\log n} \\
& \le \limsup_{n\to\infty} \frac{K(n) \log(5+4\sqrt{I}) + (2(K-1) - n_0) \log\sqrt{3+2\sqrt{I}} +  n_{K}\log(r \sqrt{3+2\sqrt{I}})}{\log n} \\
& \le \frac{\log((5+4\sqrt{I})^{1/2} (3+2\sqrt{I})r)}{\log \frac{\sqrt{5} +1}{2}} ,
\end{align*}
which proves the estimate.

\begin{proof}[Proof of Theorem~\ref{growth}]
Given $\omega \in \Omega_{a,b}$, $n>m$, $w \in \Sigma_{\alpha, \beta}$, and $\gamma$ obeying the hypothesis of the theorem, we will show that
$$\norm{M(n,m,\omega,w)} \le C (n-m) ^{\gamma}$$
for a suitable constant $C$, which implies the theorem.

As explained in \cite{DL}, we can write
$$\omega_n\dots\omega_m = xyz,$$
where $y$ is a word of length two in $a, b$ and $x^R$ (the reverse of $x$) and $z$ are prefixes of $\omega_{\alpha,\beta} $. Thus
$$\norm{M(w,\omega)} \le \norm{M(w,x)}\cdot\norm{M(w,y)}\cdot\norm{M(w,z)}.$$
Now $\norm{M(w,x^R)} = \norm{M(w,x)}$ (\cite[Lemma 5.1]{DL}). Moreover, from Lemma~\ref{mbound}, $\norm{M(w, x^R)} \le K |x|^\gamma$ and $\norm{M(w,z)} \le K |z|^\gamma $. This proves the theorem.
\end{proof}

%%%%%%%%%%%%%%%%%%%%%%%%%%%%%%%%%%%%%%%%%
% APPENDIX
%%%%%%%%%%%%%%%%%%%%%%%%%%%%%%%%%%%%%%%%%

\appendix

\section*{Acknowledgement}

W. N. Y. would like to thank Anton Gorodetski for financial support and helpful discussions.

%%%%%%%%%%%%%%%%%%%%%%%%%%%%%%%%%%%%%%%%%
% APPENDIX
%%%%%%%%%%%%%%%%%%%%%%%%%%%%%%%%%%%%%%%%%

\appendix

\section{Background on Uniform, Partial and Normal Hyperbolicity}\label{a1}

In this appendix we give a (very brief) overview of the necessary notions from the theory of hyperbolic dynamical systems. The discussion below appeared originally in \cite{Y}.

\subsection{Properties of Locally Maximal Hyperbolic Sets}\label{a_1}

A more detailed discussion can be found in \cite{Hirsch1968, Hirsch1970, Hirsch1977, Hasselblatt2002b, Hasselblatt2002}.

A closed invariant set $\Lambda\subset M$ of a diffeomorphism $f: M\rightarrow M$ of a smooth manifold $M$ is called \textit{hyperbolic} if for each $x\in\Lambda$, there exists the splitting $T_x\Lambda = E_x^s\oplus E_x^u$ invariant under the differential $Df$, and $Df$ exponentially contracts vectors in $E_x^s$ and exponentially expands vectors in $E_x^u$. If $\Lambda = M$, then $f$ is called an \textit{Anosov diffeomorphism}.

The set $\Lambda$ is called \textit{locally maximal} if there exists a neighborhood $U$ of $\Lambda$ such that
\begin{align}\label{part2_eq1}
\Lambda = \bigcap_{n\in\Z}f^n(U).
\end{align}
The set $\Lambda$ is called \textit{transitive} if it contains a dense orbit. It isn't hard to prove that the splitting $E_x^s\oplus E_x^u$ depends continuously on $x\in\Lambda$, hence $\dim(E_x^{s,u})$ is locally constant. If $\Lambda$ is transitive, then $\dim(E_x^{s,u})$ is constant on $\Lambda$. We call the splitting $E_x^s\oplus E_x^u$ a $(k_x^s, k_x^u)$ splitting if $\dim(E_x^{s,u}) = k^{s,u}$, respectively. In case $\Lambda$ is transitive, we'll simply write $(k^s, k^u)$.
\begin{definition}\label{basic_set}
We call $\Lambda\subset M$ a \textit{basic set} for $f\in\mathrm{Diff}^r(M)$, $r\geq 1$, if $\Lambda$ is a locally maximal invariant transitive hyperbolic set for $f$.
\end{definition}
Suppose $\Lambda$ is a basic set for $f$ with $(1,1)$ splitting. Then the following holds.

\subsubsection{Stability}\label{a_1_1}

Let $U$ be as in \eqref{part2_eq1}. Then there exists $\mathcal{U}\subset \mathrm{Diff}^1(M)$ open, containing $f$, such that for all $g\in\mathcal{U}$,
\begin{align}\label{part2_eq2}
\Lambda_g = \bigcap_{n\in\Z}g^n(U)
\end{align}
is $g$-invariant transitive hyperbolic set; moreover, there exists a (unique) homeomorphism $H_g:\Lambda\rightarrow\Lambda_g$ such that
\begin{align}\label{part2_eq3}
H_g\circ f|_{\Lambda} = g|_{\Lambda_g}\circ H_g.
\end{align}
Also $H_g$ can be taken arbitrarily close to the identity by taking $\mathcal{U}$ sufficiently small. In this case $g$ is said to be \textit{conjugate to} $f$, and $H_g$ is said to be \textit{the conjugacy}.

\subsubsection{Stable and Unstable Invariant Manifolds}\label{a_1_2}

Let $\epsilon > 0$ be small. For each $x\in\Lambda$ define the \textit{local stable} and \textit{local unstable} manifolds at $x$:
\begin{align*}
W_\epsilon^s(x) = \set{y\in M: d(f^n(x),f^n(y))\leq \epsilon \text{ for all }n\geq 0},
\end{align*}
\begin{align*}
W_\epsilon^u(x) = \set{y\in M: d(f^n(x),f^n(y))\leq \epsilon \text{ for all }n\leq 0}.
\end{align*}
We sometimes do not specify $\epsilon$ and write
\begin{align*}
W_\mathrm{loc}^s(x)\text{\hspace{5mm}and\hspace{5mm}}W_\mathrm{loc}^u(x)
\end{align*}
for $W_\epsilon^s(x)$ and $W_\epsilon^u(x)$, respectively, for (unspecified) small enough $\epsilon > 0$. For all $x\in\Lambda$, $W_\mathrm{loc}^{s,u}(x)$ is an embedded $C^r$ disc with $T_xW_\mathrm{loc}^{s,u}(x) = E_x^{s,u}$. The \textit{global stable} and \textit{global unstable} manifolds
\begin{align}\label{part2_eq4}
W^s(x) = \bigcup_{n\in\N}f^{-n}(W_\mathrm{loc}^s(x))\text{\hspace{5mm}and\hspace{5mm}}W^u(x) = \bigcup_{n\in\N}f^{n}(W_\mathrm{loc}^u(x))
\end{align}
are injectively immersed $C^r$ submanifolds of $M$. Define also the stable and unstable sets of $\Lambda$:
\begin{align}\label{part2_eq5}
W^s(\Lambda) = \bigcup_{x\in\Lambda}W^s(x)\text{\hspace{5mm}and\hspace{5mm}}W^u(\Lambda) = \bigcup_{x\in\Lambda}W^u(x).
\end{align}

If $\Lambda$ is compact, there exists $\epsilon > 0$ such that for any $x,y\in\Lambda$, $W_\epsilon^s(x)\cap W_\epsilon^u(y)$ consists of at most one point, and there exists $\delta > 0$ such that whenever $d(x,y) < \delta$, $x,y\in\Lambda$, then $W_\epsilon^s(x)\cap W_\epsilon^u(y)\neq \emptyset$. If in addition $\Lambda$ is locally maximal, then $W_\epsilon^s(x)\cap W_\epsilon^u(y)\in\Lambda$.

The stable and unstable manifolds $W_\mathrm{loc}^{s,u}(x)$ depend continuously on $x$ in the sense that there exists $\Phi^{s,u}:\Lambda\rightarrow\mathrm{Emb}^r(\R, M)$ continuous, with $\Phi^{s,u}(x)$ a neighborhood of $x$ along $W_\mathrm{loc}^{s,u}(x)$, where $\mathrm{Emb}^r(\R, M)$ is the set of $C^r$ embeddings of $\R$ into $M$ \cite[Theorem 3.2]{Hirsch1968}.

The manifolds also depend continuously on the diffeomorphism in the following sense. For all $g\in\mathrm{Diff}^r(M)$ $C^r$ close to $f$, define
$\Phi_g^{s,u}:\Lambda_g\rightarrow\mathrm{Emb}^r(\R,M)$ as we defined $\Phi^{s,u}$ above. Then define
\begin{align*}
\tilde{\Phi}_g^{s,u}:\Lambda\rightarrow\mathrm{Emb}^r(\R, M)
\end{align*}
by
\begin{align*}
\tilde{\Phi}_g^{s,u} = \Phi_g^{s,u}\circ H_g.
\end{align*}
Then $\tilde{\Phi}^{s,u}_g$ depends continuously on $g$ \cite[Theorem 7.4]{Hirsch1968}.

\subsubsection{Fundamental Domains}\label{a_1_2_}

Along every stable and unstable manifold, one can construct the so-called \textit{fundamental domains} as follows. Let $W^s(x)$ be the stable manifold at $x$. Let $y\in W^s(x)$. We call the arc $\gamma$ along $W^s(x)$ with endpoints $y$ and $f^{-1}(y)$ a \textit{fundamental domain}. The following holds.
\begin{itemize}
\item $f(\gamma)\cap W^s(x) = y$ and $f^{-1}(\gamma)\cap W^s(x) = f^{-1}(y)$, and for any $k\in\Z$, if $k < -1$, then $f^k(\gamma)\cap W^s(x) = \emptyset$; if $k > 1$ then $f^k(\gamma)\cap W^s(x) = \emptyset$ iff $x\neq y$;
\item For any $z\in W^s(x)$, if for some $k\in\N$, $f^k(z)$ lies on the arc along $W^s(x)$ that connects $x$ and $y$, then there exists $n\in\N$, $n\leq k$, such that $f^n(z)\in\gamma$.
\end{itemize}
Similar results hold for the unstable manifolds, after replacing $f$ with $f^{-1}$.

\subsubsection{Invariant Foliations}\label{a_1_3}

A stable foliation for $\Lambda$ is a foliation $\mathcal{F}^s$ of a neighborhood of $\Lambda$ such that
\begin{enumerate}
\item for each $x\in\Lambda$, $\mathcal{F}(x)$, the leaf containing $x$, is tangent to $E_x^s$;
\item for each $x$ sufficiently close to $\Lambda$, $f(\mathcal{F}^s(x))\subset\mathcal{F}^s(f(x))$.
\end{enumerate}
An unstable foliation $\mathcal{F}^u$ is defined similarly.

For a locally maximal hyperbolic set $\Lambda\subset M$ for $f\in\mathrm{Diff}^1(M)$, $\dim(M) = 2$, stable and unstable $C^0$ foliations with $C^1$ leaves can be constructed; in case $f\in\mathrm{Diff}^2(M)$, $C^1$ invariant foliations exist (see \cite[A.1]{Palis1993} and the references therein).

\subsubsection{Local Hausdorff and Box-Counting Dimensions}\label{a_1_4}

For $x\in\Lambda$ and $\epsilon > 0$, consider the set $W_\epsilon^{s,u}\cap\Lambda$. Its Hausdorff dimension is independent of $x\in\Lambda$ and $\epsilon > 0$.

Let
\begin{align}\label{part2_eq6}
h^{s,u}(\Lambda) = \dim_H(W_\epsilon^{s,u}(x)\cap\Lambda).
\end{align}
For properly chosen $\epsilon > 0$, the sets $W_\epsilon^{s,u}(x)\cap\Lambda$ are dynamically defined Cantor sets, so
\begin{align*}
h^{s,u}(\Lambda) < 1
\end{align*}
(see \cite[Chapter 4]{Palis1993} ). Moreover, $h^{s,u}$ depends continuously on the diffeomorphism in the $C^1$-topology \cite{Manning1983}. In fact, when $\dim(M) = 2$, these are $C^{r-1}$ functions of $f\in\mathrm{Diff}^r(M)$, for $r\geq 2$ \cite{Mane1990}.

Denote the box-counting dimension of a set $\Gamma$ by $\dim_{\mathrm{Box}}(\Gamma)$. Then
\begin{align*}
\dim_H(W^{s,u}_\epsilon(x)\cap \Lambda) = \dim_{\mathrm{Box}}(W^{s,u}_\epsilon(x)\cap \Lambda)
\end{align*}
(see \cite{Manning1983, Takens1988}).
\begin{remark}
For hyperbolic sets in dimension greater than two, many of these results do not hold in general; see \cite{Pesin1997} for more details.
\end{remark}

\subsection{Partial Hyperbolicity}\label{a_2}

For a more detailed discussion, see \cite{Pesin2004, Hasselblatt2006}.

An invariant set $\Lambda\subset M$ of a diffeomorphism $f\in\mathrm{Diff}^r(M)$, $r\geq 1$, is called \textit{partially hyperbolic (in the narrow sense)} if for each $x\in\Lambda$ there exists a splitting $T_xM = E_x^s\oplus E_x^c\oplus E_x^u$ invariant under $Df$, and $Df$ exponentially contracts vectors in $E_x^s$, exponentially expands vectors in $E_x^u$, and $Df$ may contract or expand vectors from $E_x^c$, but not as fast as in $E_x^{s,u}$. We call the splitting $(k_x^s, k_x^c, k_x^u)$ splitting if $\dim(E_x^{s,c,u}) = k_x^{s,c,u}$, respectively. We'll write $(k^s,k^c,k^u)$ if the dimension of subspaces does not depend on the point.

\subsection{Normal Hyperbolicity}\label{a_3}

For a more detailed discussion and proofs see \cite{Hirsch1977} and also \cite{Pesin2004}.

Let $M$ be a smooth Riemannian manifold, compact, connected and without boundary. Let $f\in\mathrm{Diff}^r(M)$, $r\geq 1$. Let $N$ be a compact smooth submanifold of $M$, invariant under $f$. We call $f$ \textit{normally hyperbolic} on $N$ if $f$ is partially hyperbolic on $N$. That is, for each $x\in N$,
\begin{align*}
T_xM = E_x^s\oplus E_x^c\oplus E_x^u
\end{align*}
with $E_x^c = T_xN$. Here $E_x^{s,c,u}$ is as in Section \ref{a_2}. Hence for each $x\in N$ one can construct local stable and unstable manifolds $W_\epsilon^s(x)$ and $W_\epsilon^u(x)$, respectively, such that
\begin{enumerate}
\item $x\in W_\mathrm{loc}^s(x)\cap W_\mathrm{loc}^u(x)$;
\item $T_x W_\mathrm{loc}^s(x) = E^s(x)$, $T_xW_\mathrm{loc}^u(x) = E^u(x)$;
\item for $n\geq 0$,
\begin{align*}
d(f^n(x),f^n(y))\xrightarrow[n\rightarrow\infty]{}0\text{ for all }y\in W_\mathrm{loc}^s(x),
\end{align*}
\begin{align*}
d(f^{-n}(x),f^{-n}(y))\xrightarrow[n\rightarrow\infty]{}0\text{ for all }y\in W_\mathrm{loc}^u(x).
\end{align*}
\end{enumerate}
(For the proof see \cite[Theorem 4.3]{Pesin2004}).
These can then be extended globally by
\begin{align*}
W^s(x)& = \bigcup_{n\in\N}f^{-n}f(W^s_\mathrm{loc}(x))\\
W^u(x)& = \bigcup_{n\in\N}f^n(W^u_\mathrm{loc}(x))
\end{align*}
Set
\begin{align*}
W_\mathrm{loc}^{cs}(N) = \bigcup_{x\in N}W_\mathrm{loc}^s(x)\text{\hspace{5mm}and\hspace{5mm}} W_\mathrm{loc}^{cu}(x) = \bigcup_{x\in N}W_\mathrm{loc}^u(x).
\end{align*}
\begin{theorem}[Hirsch, Pugh and Shub \cite{Hirsch1977}]\label{a1_thm1}
The sets $W_\mathrm{loc}^{cs}(N)$ and $W_\mathrm{loc}^{cu}(N)$, restricted to a neighborhood of $N$, are smooth submanifolds of $M$. Moreover,
\begin{enumerate}
\item $W^\mathrm{cs}_\mathrm{loc}(N)$ is $f$-invariant and $W^\mathrm{cu}_\mathrm{loc}$ is $f^{-1}$-invariant;
\item $N = W^\mathrm{cs}_\mathrm{loc}(N)\bigcap W^\mathrm{cu}_\mathrm{loc}(N)$;
\item For every $x\in N$, $T_x W^\mathrm{cs, cu}_\mathrm{loc}(N) = E_x^{s,u}\oplus T_xN$;
\item $W^\mathrm{cs}_\mathrm{loc}(N)$ ($W^\mathrm{cu}_\mathrm{loc}(N)$) is the only $f$-invariant ($f^{-1}$-invariant) set in a neighborhood of $N$;
\item $W^\mathrm{cs}_\mathrm{loc}(N)$ (respectively, $W^\mathrm{cu}_\mathrm{loc}(N)$) consists precisely of those points $y\in M$ such that for all $n\geq 0$ (respectively, $n\leq 0$), $d(f^n(x),f^n(y)) < \epsilon$ for some $\epsilon > 0$.
\item $W^\mathrm{cs,cu}_\mathrm{loc}(N)$ is foliated by $\set{W_\mathrm{loc}^\mathrm{s,u}(x)}_{x\in N}$.
\end{enumerate}
\end{theorem}

%%%%%%%%%%%%%%%%%%%%%%%%%%%%%%%%%%%%%%%%%%%%
% BIBLIOGRAPHY
%%%%%%%%%%%%%%%%%%%%%%%%%%%%%%%%%%%%%%%%%%%%


\begin{thebibliography}{10}

\bibitem{BR} M.\ Baake, J.\ A.\ G.\ Roberts, {Reversing symmetry group of $\mathrm{GL}(2,\Z)$ and $\mathrm{PGL}(2,\Z)$ matrices with connections to cat maps and trace maps}, \textit{J. Phys. A: Math. Gen.}\ \textbf{30} (1997), 1549--1573.

\bibitem{Bedford1993} E.\ Bedford, M.\ Lyubich, J.\ Smillie, Polynomial Diffeomorphisms of $\C^2$. IV: The measure of maximal entropy and laminar currents, \textit{Invent. Math.} \textbf{112} (1993) 77--125.

\bibitem{C} S.\ Cantat, Bers and H\'enon, Painlev\'e and Schr\"odinger, \textit{Duke Math.\ J.}\ \textbf{149} (2009), 411--460.

\bibitem{Casdagli1986} M.\ Casdagli, Symbolic dynamics for the renormalization map of a quasiperiodic Schr\"odinger equation, \textit{Commun.\ Math.\ Phys.}\ \textbf{107} (1996), 295--318.

\bibitem{Damanik2000} D.\ Damanik, Strictly ergodic subshifts and associated operators, in Spectral theory and mathematical physics: a Festschrift in honor of Barry Simon's 60th birthday, \textit{Sympos. Pure Math.}, \textbf{76}, Part 2, Amer. Math. Soc., Providence, RI., (2007), 505--538.

\bibitem{DEGT} D.\ Damanik, M.\ Embree, A.\ Gorodetski, S.\ Tcheremchantsev, The fractal dimension of the spectrum of the
Fibonacci Hamiltonian, \textit{Commun.\ Math.\ Phys.}\ \textbf{280} (2008), 499--516.

\bibitem{Damanik2009} D.\ Damanik, A.\ Gorodetski, Hyperbolicity of the trace map for the weakly coupled Fibonacci Hamiltonian, \textit{Nonlinearity}\ \textbf{20} (2009), 123--143.

\bibitem{DG} D.\ Damanik, A.\ Gorodetski, Spectral and quantum dynamical properties of the weakly coupled Fibonacci Hamiltonian, \textit{Commun.\ Math.\ Phys.}\ \textbf{305} (2011), 221-277.

\bibitem{DL} D.\ Damanik, D.\ Lenz, Uniform spectral properties of one-dimensional quasicrystals, I.~Absence of eigenvalues,
\textit{Commun.\ Math.\ Phys.}\ \textbf{207} (1999), 687--696.

\bibitem{DL3} D.\ Damanik, D.\ Lenz, Uniform spectral properties of one-dimensional quasicrystals, II.~The Lyapunov exponent, \textit{Lett.\ Math.\ Phys.}\ \textbf{50} (1999), 245--257.

\bibitem{DL2} D.\ Damanik, D.\ Lenz, Uniform Szeg\H{o} cocycles over strictly ergodic subshifts, \textit{J.\ Approx.\ Theory} \textbf{144} (2007), 133-138.

\bibitem{DT} D.\ Damanik, S.\ Tcheremchantsev, Power-law bounds on transfer matrices and quantum dynamics in one dimension, \textit{Commun.\ Math.\ Phys.} {\bf 236} (2003), 513--534.

\bibitem{delaLlave1986} R.\ de la Llave, M.\ Marco, R.\ Moriyon, Canonical perturbation theory of Anosov systems and regularity results for the Livsic cohomology equation, \textit{Ann.\ of Math.}\ \textbf{123} (1986), 537--611.

\bibitem{GorodetskiXXXX} A.\ Gorodetski, R.\ de la Llave, L.\ Wong, Analyticity of the dimension of the density of states measure of the weakly coupled Fibonacci Hamiltonian (in preparation).

\bibitem{Hasselblatt2002} B.\ Hasselblatt, \textit{Handbook of Dynamical Systems: Hyperbolic Dynamical Systems}, vol. 1A, Elsevier B.\ V.,\ Amsterdam, The Netherlands (2002).

\bibitem{Hasselblatt2002b} B.\ Hasselblatt, A.\ Katok, \textit{Handbook of Dynamical Systems: Principal Structures}, vol.~1A, Elsevier B.\ V.,\ Amsterdam, The Netherlands (2002).

\bibitem{Hasselblatt2006} B.\ Hasselblatt, Ya.\ Pesin, \textit{Handbook of Dynamical Systems: Partially Hyperbolic Dynamical Systems}, vol.~1B, Elsevier B.\ V.,\ Amsterdam, The Netherlands (2006).

\bibitem{Hirsch1970} M.\ Hirsch, J.\ Palis, C.\ Pugh, M.\ Shub, Neighborhoods of hyperbolic sets, \textit{Invent.\ Math.}\ \textbf{9} (1970), 121--134.

\bibitem{Hirsch1968} M.\ W.\ Hirsch, C.\ C.\ Pugh, Stable manifolds and hyperbolic sets, \textit{Proceedings of Symposia in Pure Mathematics} \textbf{14} (1968), 133--163.

\bibitem{Hirsch1977} M.\ W.\ Hirsch, C.\ C.\ Pugh, M.\ Shub, \textit{Invariant Manifolds}, Lecture Notes in Math.\ \textbf{583}, Springer-Verlag (1977).

\bibitem{IT} B.\ Iochum, D.\ Testard, Power law growth for the resistance in the Fibonacci model, \textit{J.\ Stat.\ Phys.}\ \textbf{65} (1991), 715--723.

\bibitem{KKL} R.\ Killip, A.\ Kiselev, Y.\ Last, Dynamical upper bounds on wavepacket spreading, \textit{Amer.\ J.\ Math.}\ \textbf{125} (2003), 1165--1198.

\bibitem{kkt} M.\ Kohmoto, L.\ Kadanoff, C.\ Tang, Localization problem in one dimension: Mapping and escape, \textit{Phys.\ Rev.\ Lett.} {\bf 50} (1983), 1870--1872.

\bibitem{Mane1990} R.\ Ma\~{n}\'{e}, The Hausdorff dimension of horseshoes of diffeomorphisms of surfaces, \textit{Boletim da Sociadade Brasileira de Mathem\'{a}tica} \textbf{20} (1990), 1--24.

\bibitem{Manning1983} A.\ Manning, M.\ McCluskey, Hausdorff dimension for horseshoes, \textit{Ergod. Th. \& Dynam. Sys.} \textbf{3} (1983), 251--260.

\bibitem{oprss} S.\ Ostlund, R.\ Pandit, D.\ Rand, H.\ Schellnhuber, E.\ Siggia, One-dimensional Schr\"odinger equation with an almost periodic potential, \textit{Phys.\ Rev.\ Lett.}\ \textbf{50} (1983), 1873--1877.

\bibitem{Palis1993} J.\ Palis, F.\ Takens, \textit{Hyperbolicity and Sensetive Chaotic Dynamics at Homoclinic Bifurcations}, Cambridge University Press, Cambridge (1993).

\bibitem{Pesin1997} Ya.\ Pesin, \textit{Dimension Theory in Dynamical Systems}, Chicago Lectures in Mathematics Series (1997).

\bibitem{Pesin2004} Ya.\ Pesin, \textit{Lectures on Partial Hyperbolicity and Stable Ergodicity}, Z\"urich Lectures in Advanced Mathematics, European Mathematical Society (2004).

\bibitem{R} L.\ Raymond, A constructive gap labelling for the discrete Schr\"odinger operator on a quasiperiodic chain, Preprint (1997).

\bibitem{Rob} J.\ A.\ G.\ Roberts, Escaping orbits in trace maps, \textit{Physica A} \textbf{228} (1996), 295--325.

\bibitem{BR2} J.\ A.\ G.\ Roberts, M.\ Baake, Trace maps as 3D reversible dynamical systems with an invariant, \textit{J.\ Stat.\ Phys.}\ \textbf{74} (1994), 829--888.

\bibitem{BR3} J.\ A.\ G.\ Roberts, M.\ Baake, The dynamics of trace maps, \textit{Hamiltonian Mechanics: Integrability and Chaotic Behavior}, Ed.~J.~Seimenis, NATO ASI Series B: Physics (Plenum Press, New York) (1994), 275--285.

\bibitem{S} B.\ Simon, \textit{Orthogonal Polynomials on the Unit Circle. Part~1. Classical Theory}, Colloquium Publications \textbf{54}, American Mathematical Society, Providence (2005).

\bibitem{S2} B.\ Simon, \textit{Orthogonal Polynomials on the Unit Circle. Part~2. Spectral Theory}, Colloquium Publications \textbf{54}, American Mathematical Society, Providence (2005).

\bibitem{Smale1967} S.\ Smale, Differentiable dynamical systems, \textit{Bull.\ Amer.\ Math.\ Soc.}\ \textbf{73} (1967), 747--817.

\bibitem{S87} A.\ S\"ut\H{o}, The spectrum of a quasiperiodic Schr\"odinger operator, \textit{Commun.\ Math.\ Phys.}\ \textbf{111} (1987), 409--415.

\bibitem{Takens1988} F.\ Takens, Limit capacity and Hausdorff dimension of dynamically defined Cantor sets, \textit{Dynamical Systems}, Lecture Notes in Mathematics \textbf{1331} (1988), 196--212.

\bibitem{Y} W.\ N.\ Yessen, On the spectrum of 1D quantum Ising quasicrystal, Preprint (arXiv:1110.6894).

\bibitem{Y2} W.\ N.\ Yessen, Spectral analysis of tridiagonal Fibonacci Hamiltonians, \textit{J. Spectr. Theory} (to appear) (2012).

\bibitem{Y3} W.\ N.\ Yessen, Properties of 1D classical and quantum Ising quasicrystals: rigorous results, Preprint (arXiv:1203.2221v2).

\end{thebibliography}
\end{document}